\documentclass[11pt]{amsart}
\usepackage{amsmath}
\usepackage{hyperref} 
\usepackage{amsfonts}
\usepackage{amssymb}
\usepackage{amsthm}
\usepackage[arrow, matrix, curve]{xy} 
\usepackage{hyperref}
\usepackage{amssymb,amscd,url,tikz,stmaryrd,mathtools,fixltx2e}
\usepackage{amssymb,amscd,url}

\usepackage{color}

\newcommand{\A}{{\mathbb A}}
\newcommand{\Z}{{\mathbb Z}}
\renewcommand{\S}{{\mathbb S}}
\newcommand{\R}{{\mathbb R}}
\newcommand{\C}{{\mathbb C}}

\newcommand{\N}{{\mathbb N}}
\renewcommand{\H}{{\mathbb H}}

\def\Res{{\,\rm Res}}
\def\Re{{\rm Re}}
\def\Im{{\rm Im}}
\def\Id{{\rm Id}}
\def\ii{{\rm i}}
\def\xx{{\bf x}}
\def\sl{\mathfrak{sl}}
\def\su{\mathfrak{su}}

\def\tr{{\rm trace}}
\def\Area{{\rm Area}}
\def\SL{{\rm SL}}
\def\SU{{\rm SU}}
\def\z{\overline{z}}
\renewcommand{\matrix}[1]{\left(\begin{array}{cc} #1\end{array}\right)}
\newcommand{\minimatrix}[1]{\left(\begin{smallmatrix}#1\end{smallmatrix}\right)}
\newcommand{\wt}[1]{\widetilde{#1}}
\newcommand{\wh}[1]{\widehat{#1}}
\newcommand{\cal}[1]{{\mathcal #1}}
\theoremstyle{plain}
\newtheorem{theorem}{Theorem}
\newtheorem{lemma}{Lemma}
\newtheorem{proposition}[lemma]{Proposition}
\newtheorem{remark}[lemma]{Remark}
\newtheorem{corollary}[lemma]{Corollary}

\newtheorem{example}[lemma]{Example}

\begin{document}

\title{Area Estimates for High genus Lawson surfaces via DPW}

\author{Lynn Heller}

 \author{Sebastian Heller}
 
\author{Martin Traizet}

\noindent
\address{Institut f\"ur Differentialgeometrie\\
Welfengarten 1\\
30167 Hannover\\
Germany} 
\email{lynn.heller@math.uni-hannover.de}
 \noindent 
\address{Institut f\"ur Differentialgeometrie\\
Welfengarten 1\\
30167 Hannover\\
Germany} 
\email{seb.heller@gmail.com}
\noindent
\address{Institut Denis Poisson, CNRS UMR 7350 \\
Facult\'e des Sciences et Techniques \\
Universit\'e de Tours }
\email{martin.traizet@univ-tours.fr }

\begin{abstract}
Starting at a saddle tower surface, we give a new existence proof of the Lawson surfaces $\xi_{m,k}$ of high genus by deforming the corresponding DPW potential. As a byproduct, we obtain for fixed $m$ estimates on the area of  $ \xi_{m,k}$ in terms of their genus $g=m k \gg1$.
\end{abstract}

\thanks{We would like to thank Laurent Mazet for providing us with the monotonicity formula argument in the proof of Corollary 20. The first author is supported by the  {\em Deutsche Forschungsgemeinschaft} within the priority program {\em Geometry at Infinity}. The second author is funded by the {\em Deutsche Forschungsgemeinschaft} via GRK 1670.}
\maketitle
\setcounter{tocdepth}{1}
\tableofcontents
\section*{Introduction}\label{sec:intro}
\noindent
Minimal surfaces are important objects in differential geometry which have fascinated geometers for centuries. Depending 
 on the curvature of the ambient space, different techniques were developed to prove existence, uniqueness (possibly under certain 
 geometric constraints), and to study the space of 
 embedded minimal surfaces. In Euclidean space, minimal surfaces can be explicitly 
 parametrised via Weierstrass representation. Constructing minimal surfaces in a compact symmetric space -- such as the 
 round 3-sphere -- is much more involved. \\

\noindent
Examples of compact embedded minimal surfaces in the $3$-sphere of all genera were first found by Lawson \cite{Lawson} using the solution of the Plateau problem with respect to a polygonal boundary curve. Though enormous achievements have been made in the theory of minimal surfaces in positively curved 3-manifolds by Min-Max theory in recent years (see for example \cite{MarNev18} and references therein), we still lack knowledge about the simplest compact minimal surfaces of genus $g\geq2$ in the round $3$-sphere. For example, the area of these surfaces is still unknown and the index and stability for Lawson $\xi_{1,g}$-surfaces were only recently computed 
\cite{KapWiy}. \\

\noindent
It is well known that the Lawson surfaces $\xi_{m,k}$ converge for fixed $m$ and $k \rightarrow \infty$ to the union of $m+1$ great spheres intersecting in a great circle.
In this paper, we go backwards and construct Lawson surfaces $\xi_{m,k}$ for $k\gg 1$ by desingularizing the union of $m+1$ great spheres using a Karcher saddle tower, a minimal surface generalizing the classical Scherk surface (see Section \ref{sec:karcher}).
As a consequence,  our analysis determines the asymptotic behaviour of the area of the Lawson surfaces $\xi_{m,k}$ for $k\gg 1$ up to second order.
In particular, in the case $m=1$, Theorem \ref{the:area} gives 
\[\Area(\xi_{1,g})=8\pi \left(1-\frac{\ln 2}{2g}+\frac{\ln 2}{2g^2}+ O\left(\frac{1}{g^3}\right)\right)\]
for the area of the Lawson surface $\xi_{1,g}$ of genus $g$ with $g\gg 1$.
The Lawson surfaces  $\xi_{1,g}$ are conjectured to minimize the Willmore energy for surfaces of genus $g$ (\cite{kusner}, Conjecture 8.4). Since the area of a minimal surface in $\S^3$ is its Willmore energy, the above equation yields estimates for the conjectured minimum Willmore energy of compact surfaces of genus $g\gg 1$. In \cite{KLS}, the large genus limit of the minimal Willmore energy is shown to be $8 \pi$, giving some evidence to the Kusner conjecture.
\medskip

\noindent
Desingularization is a well established and productive method to construct minimal surfaces in various spaces using PDE methods (see for example \cite{KapYan}
for an example of such construction in the 3-sphere). However, these methods would not give such fine area estimates as the ones that we obtain in this paper. We shall carry out the construction using integrable system methods, which in essence allow for more explicit formulas.
\medskip

\noindent
In this paper we consider 
 a conformally parametrised minimal immersion $f$ from a Riemann surface $\Sigma$ into the round $3$-sphere. The harmonicity of $f$  gives rise to a symmetry of the Gauss-Codazzi equations in the 3-sphere inducing an associated family of (isometric) minimal surfaces on the universal covering of $\Sigma$ with rotated Hopf differential. This family of surfaces is the geometric counterpart of an associated $\C_*$-family 
of flat $\text{SL}(2,\mathbb C)$-connections $\nabla^\lambda$  \cite{hitchin} on the trivial $\C^2$-bundle over $\Sigma$ satisfying 
\begin{enumerate}\label{closingconditions}
\item[(i)] conformality:  $\nabla^\lambda=\lambda^{-1}\Phi+\nabla+\lambda \Psi$ for a nilpotent $\Phi\in \Omega^{1,0}(\Sigma,\mathfrak{sl}(2,\mathbb C)) ;$
\item[(ii)]  intrinsic closing: $\nabla^\lambda$ is unitary for all $\lambda\in\S^1,$ i.e., $\nabla$ is unitary and $\Psi=\Phi^*$ with respect to the standard hermitian metric on $\underline{\C}^2$;
\item[(iii)] extrinsic closing: $\nabla^\lambda$ is trivial for $\lambda=\pm1.$
\end{enumerate}
The minimal surface can be reconstructed from the associated family of connections as the gauge between
$\nabla^{-1}$ and $\nabla^1.$ Constructing minimal surfaces is thus equivalent to writing down appropriate families of flat connections.\\

\noindent
 The DPW method \cite{DPW}, which can be viewed as a generalisation of the Weierstrass representation for minimal surfaces in Euclidean space,  is a way to generate families of flat connections from
so-called {\em DPW potentials} on $\Sigma$, denoted by  $\eta = \eta^\lambda$, using loop group factorisation methods.
We summarise the basic procedure in Section \ref{sec:DPW}. On simply connected domains $\Sigma$, all DPW potentials give rise to minimal surfaces. Whenever the domain has non-trivial topology, finding DPW potentials satisfying  
conditions equivalent to (i),(ii) and (iii) is difficult. %We refer to these conditions as solving the {\em monodromy problem}, see Section \ref{rem:closing}.
So far, only special surface classes, such as trinoids \cite{SKKR}, tori, and more recently  $n$-noids were constructed using DPW \cite{nnoids,minoids}. In this paper we give the first existence proof of  closed embedded  minimal surfaces of high genus in the 3-sphere via DPW.\\

\noindent
The outline of the paper is as follows.
We start with recalling the classical construction of Lawson surfaces, the Weierstrass representation of Karcher saddle tower surfaces, and some general facts concerning loop groups and DPW method in Section \ref{sec:basics}.
In Section \ref{sec:tech}, we propose a family of  DPW potentials for Lawson surfaces.
Because of symmetries, Lawson surface $\xi_{m,k}$ is a $(k+1)$-sheeted branched cover of the Riemann sphere. We choose the DPW potential $\eta$ to be well defined on the Riemann sphere, with simples poles at the branch points of the covering.
Our potential $\eta=\eta_t$ actually depends on a small real parameter $t$ and
closed minimal surfaces are recovered when $t=\frac{1}{2k+2}$.
The Monodromy Problem is solved using the Implicit Function Theorem at $t=0$.
The strategy here is analogous to \cite{nnoids}, \cite{minoids} and similar to \cite{HHS}.
In the DPW setup, the area of a minimal surface can be computed explicitly from the DPW potential, see Corollary \ref{cor:areares}. Thus  we compute the time derivative of $\eta_t$ at $t=0$  up to order 2 in Section \ref{sec:der}. The constructed family of surfaces are identified to be the Lawson surfaces $\xi_{m,k}$ in Section \ref{sec:construction} for $k$ sufficiently large. Finally, 
using the derivatives of $\eta_t$ computed in Section \ref{sec:der} we obtain an asymptotic expansion of the area of high genus Lawson surfaces.

\section{Preliminary}\label{sec:basics}
\noindent
In order to fix notations and to be self-contained
we shortly recall the construction of Lawson surfaces, saddle towers as well as general facts about loop groups and DPW.
\subsection{Lawson surfaces} $\; $\\
The original construction of the Lawson surfaces  \cite{Lawson}
$$\xi_{m,k} : \Sigma \longrightarrow \S^3$$
  uses the solution to the Plateau problem. Consider two orthogonal great circles $C_1$ and $C_2$ in the round 3-sphere. Let $P_1,..,P_{2m+2}$ denote $(2m+2)$ equidistant points  on $C_1$,
and $Q_1,..,Q_{2k+2}$ denote $(2k+2)$ equidistant points on $C_2$. For the convex geodesic polygon
\[\overline{P_1Q_1P_2Q_2}\]
the corresponding Plateau solution, see Figure \ref{fig:plateau}, is a minimal surface in $S^3$. A closed minimal surface is obtained from this fundamental piece by repeatedly reflecting it across its geodesic boundaries. The resulting surfaces are called Lawson surfaces $\xi_{m,k}$, are embedded and of genus $g=m\cdot k$.\\
 \begin{figure}[h]
\centering
\includegraphics[width=0.375\textwidth]{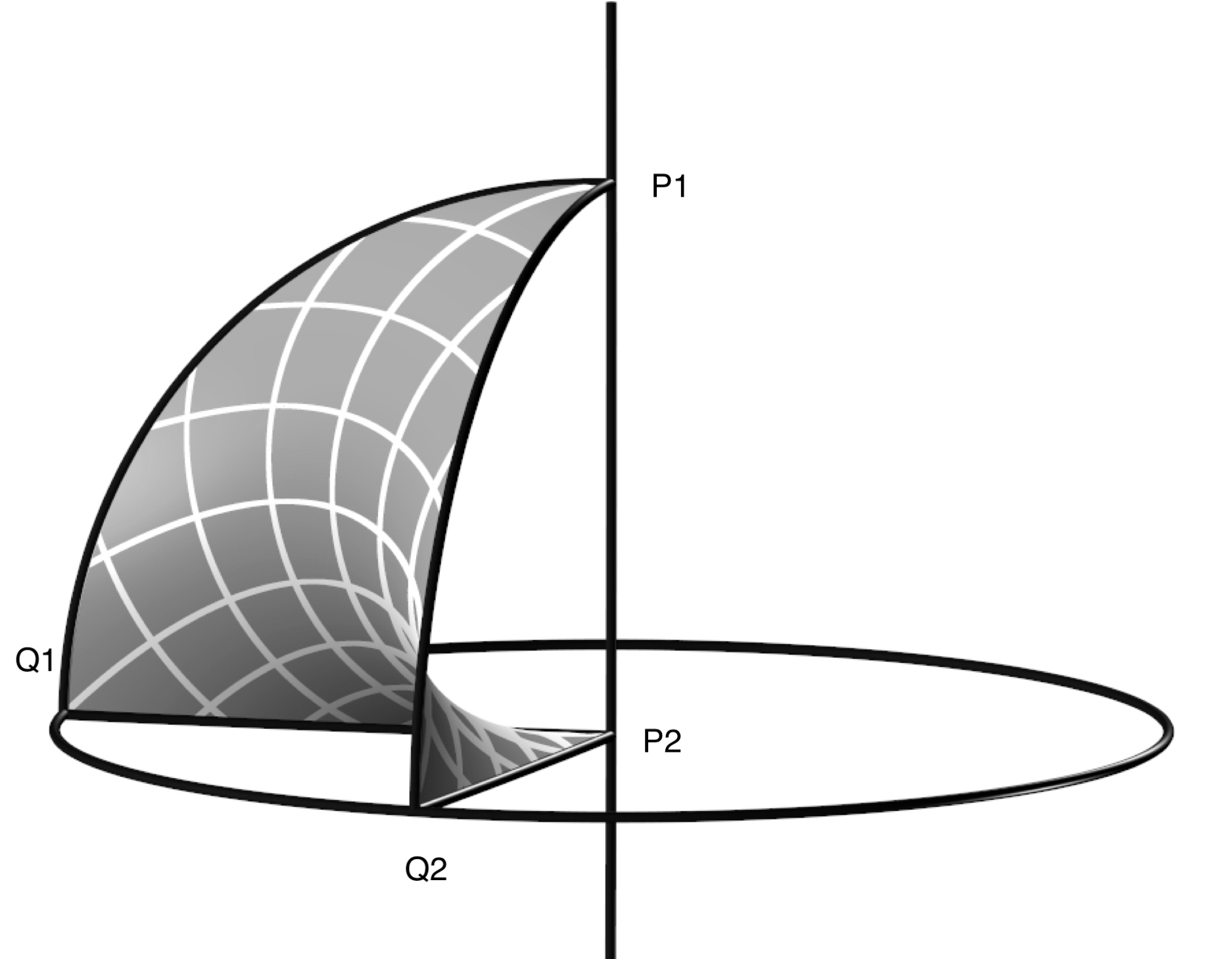}
\includegraphics[width=0.375\textwidth,angle=-1]{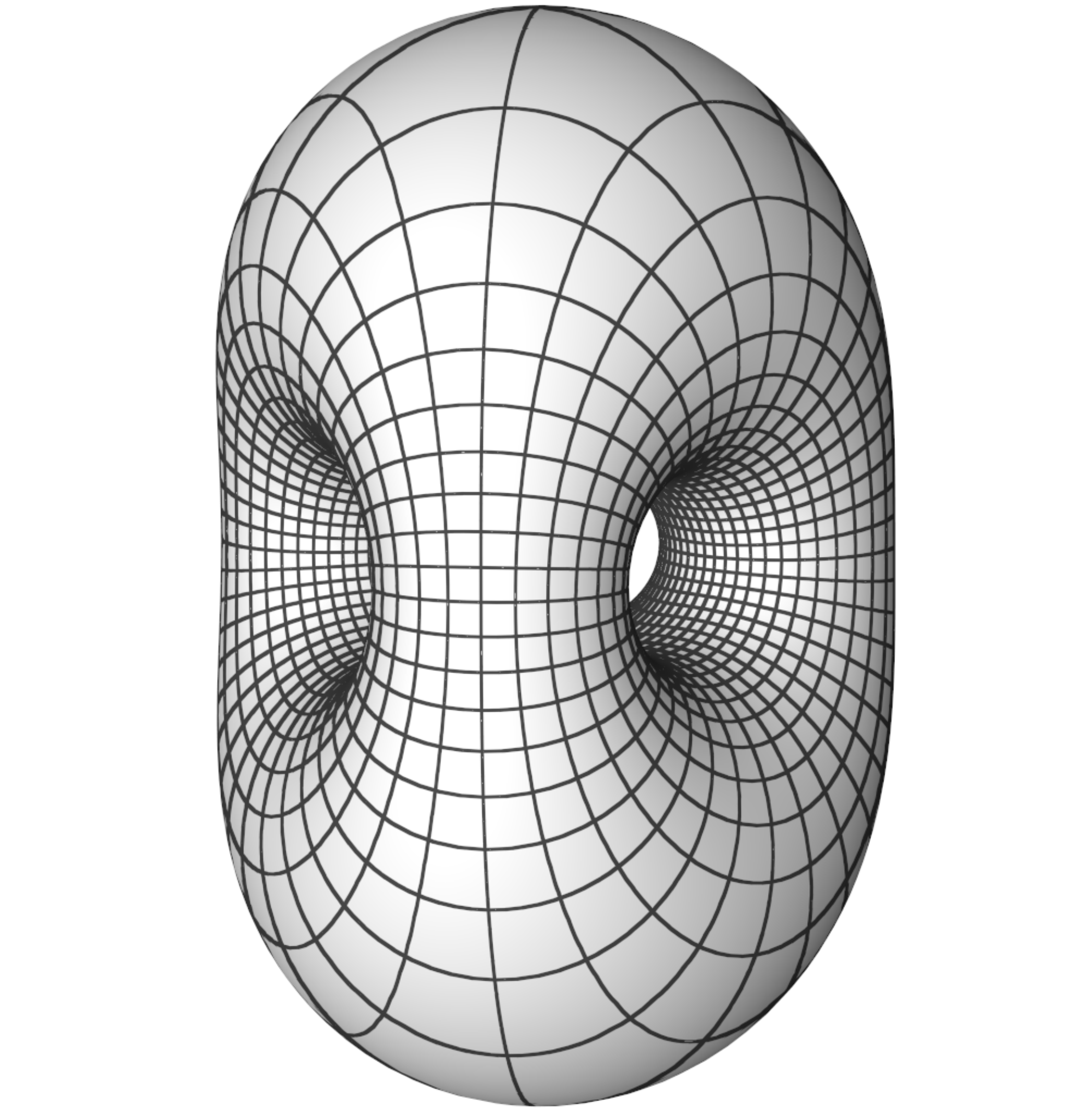}
\hspace{0.25cm}
\caption{
\footnotesize
The Plateau solution of a geodesic 4-gon in the 3-sphere and the Lawson surface of genus 2, stereographically projected to the Euclidean space. 
Images by Nicholas Schmitt with xLab.
}
\label{fig:plateau}
\end{figure}

\noindent
By construction the Lawson surfaces possess a large symmetry group. The subgroup of orientation preserving symmetries (both on the surface and in 3-space)
contains

\[\Z_{m+1}\times\Z_{k+1},\]
where the action is the natural rotation in the planes spanned by the circles $C_1$ and  $C_2,$ respectively.\\

\noindent
The minimal surface $\xi_{m,k}$ induces a Riemann surface structure on $\Sigma.$ The quotient of the Riemann surface by
the symmetries $\Z_{m+1}$ and  $\Z_{k+1},$ respectively, is $\C P^1$ and the covering $\Sigma \rightarrow \C P^1$ is totally branched over $2k+2$ and respectively $2m+2$
points. Using the additional reflection symmetries, these 
$2k+2$ and respectively $2m+2$
points are in equidistance on the unit circle of the (round) 2-sphere.
\begin{remark}
Since the surfaces $\xi_{m,k}$ and $\xi_{k,m}$ are isometric,  the Lawson surfaces $\xi_{k,k}$ admit an additional orientation preserving symmetry.
\end{remark}

\noindent
All Lawson surfaces admit additional symmetries which  are not orientation preserving in space or not orientation preserving on the surfaces. They are given by reflections across geodesics  contained in the surfaces (e.g., the polygonal boundary of the fundamental piece) or  by reflection across geodesic 2-spheres which intersect the surface orthogonally, e.g., symmetry planes of the geodesic polygon.

 \begin{figure}
\centering
\includegraphics[width=0.375\textwidth,angle=-1]{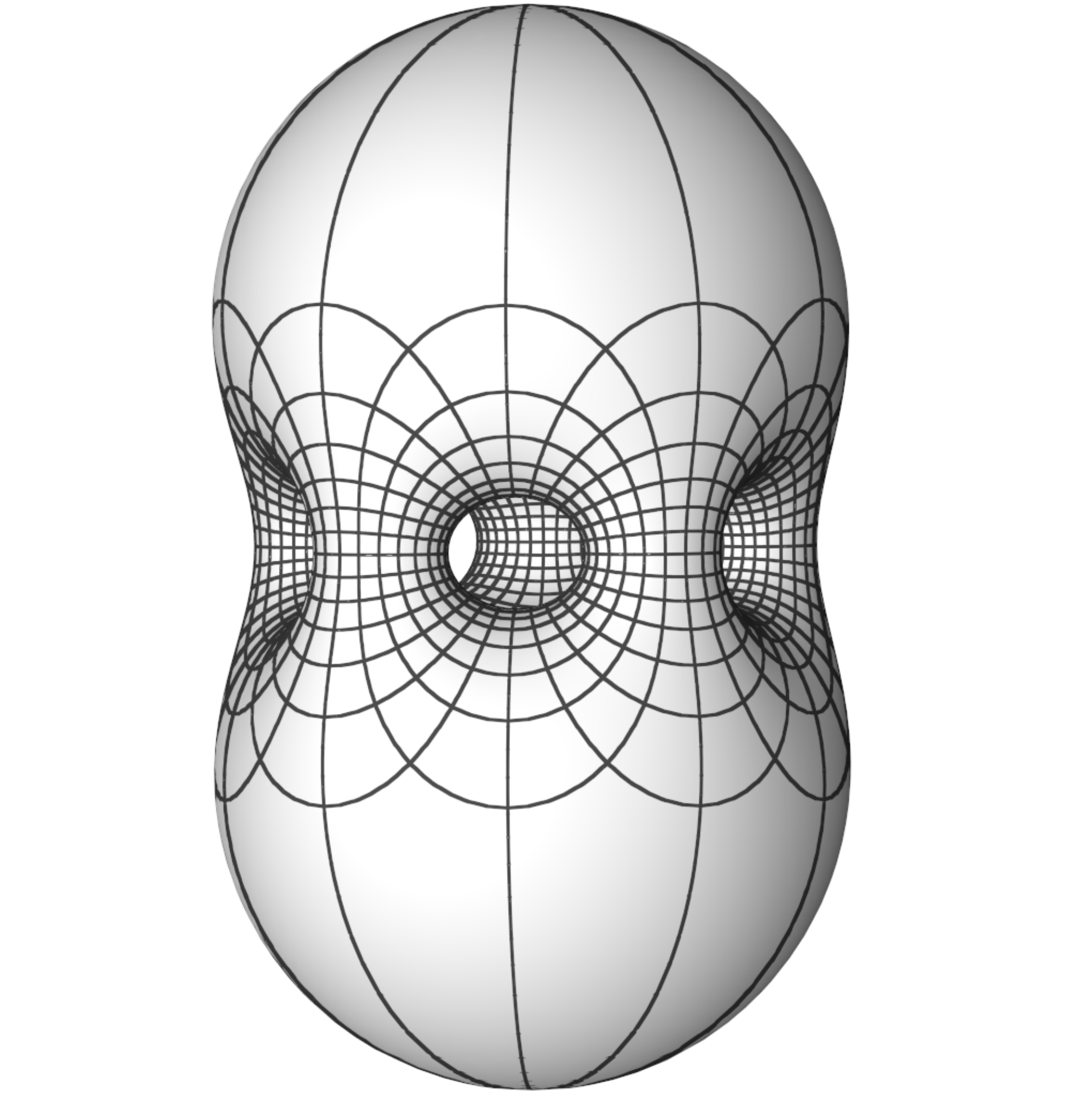}
\includegraphics[width=0.375\textwidth,angle=-1]{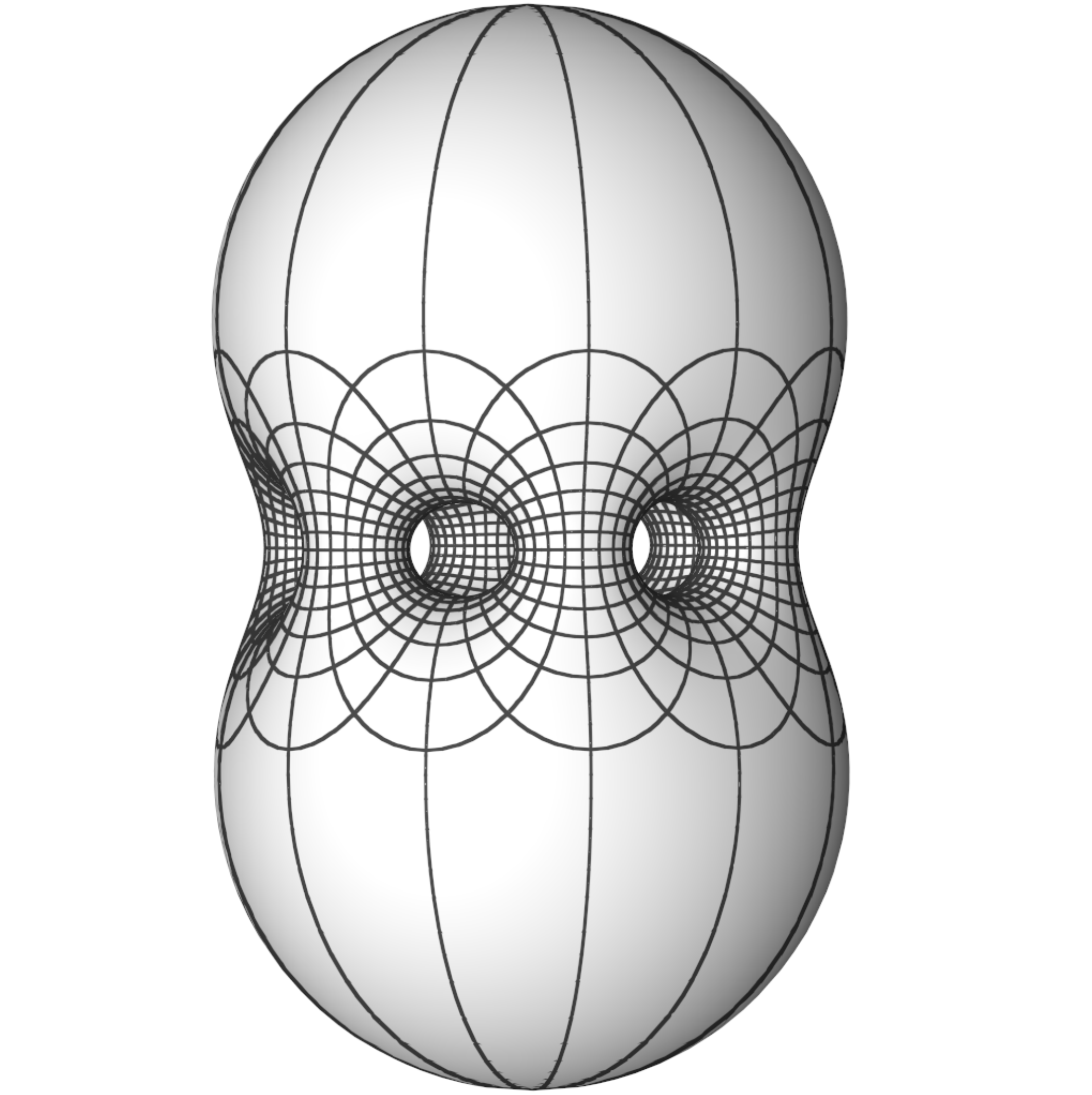}
\caption{
\footnotesize
The Lawson surfaces of genus 4 and 5.
}
\label{fig:l5}
\end{figure}

\subsection{Saddle Tower Surfaces}\label{sec:karcher}$\; $\\
\noindent
 Karcher \cite{Kar} generalised Scherk's singly periodic surface 
 to surfaces with  $n=2m+2$ Scherk type ends and constant angle $\frac{2\pi}{2m+2}$ between consecutive ends, see the  figure  in \cite{Kar}. These surfaces are called saddle tower surfaces and their Weierstrass data 
are given by

\begin{equation}\label{saddletower}
g=\frac{\ii}{z^m} \quad \text{ and } \quad
\omega=\frac{2n z^{2m}\,dz}{z^{2m+2}-1}.
\end{equation} 
\subsection{Loop groups}$\; $\\
In the following we give a comprehensive introduction to the theory of loop groups which contains only relevant theorems and facts with regard to the paper. For details we refer to \cite{PrSe}.
Let  $G$ be a finite dimensional real Lie group with Lie algebra $\mathfrak g$.  We define the loop spaces
\begin{itemize}
\item $\Lambda G:= \{$ real analytic maps (loops) $\Phi\colon \S^1\longrightarrow G,\quad \lambda \longmapsto \Phi^\lambda\};$

\item $\Lambda \mathfrak{g}:=\{$ real analytic maps (loops) $\eta\colon \S^1\longrightarrow\mathfrak{g},\quad  \lambda \longmapsto \eta^\lambda\}.$ 
\end{itemize}
$\Lambda G$ is an infinite dimensional Frechet Lie group via pointwise multiplication with $\Lambda \mathfrak{g}$ as its
Lie algebra. For a complex Lie group $G^\C$ we denote

\[\Lambda_+G^\C=\{\Phi\in\Lambda G^\C\mid \Phi \text{ extends holomorphically to } \lambda=0\}\]

\noindent
and

\[\Lambda_+\mathfrak g^\C=\{\eta\in\Lambda \mathfrak g^\C\mid \eta \text{ extends holomorphically to } \lambda=0\}.\]

\noindent
In the particular case of $G^\C=\text{SL}(2,\C)$ we denote

\[\mathcal B=\{B\in\SL(2,\C)\mid B \text{ is upper triangular with positive diagonal entries}\}\]

\noindent
and

\[\Lambda_+^\R\text{SL}(2,\C)=\{B\in\Lambda_+\text{SL}(2,\C)\mid B(0) \in \mathcal B\}.\]

\noindent
We will make use of the following theorem, often referred to as Iwasawa decomposition:

\begin{theorem}[\cite{PrSe}]\label{Iwasawa}
Let $\Phi\in\Lambda\text{SL}(2,\C)$. Then there exist a splitting 
 $$\Phi=F\cdot B$$
 with $F\in \Lambda \text{SU}(2)$ and $B \in \Lambda_+^\R\text{SL}(2,\C).$
This splitting is unique and depends real analytically on $\Phi.$ The pair $(F,B)$ is called the Iwasawa decomposition
of $\Phi$.\\
\end{theorem}

\subsection{The DPW method}\label{sec:DPW}$\; $\\
Let $\Sigma$ be a Riemann surface. 
A DPW potential on $\Sigma$ is a closed 1-form

\[\eta\in\Omega^{1,0}(\Sigma,\Lambda \mathfrak{sl}(2,\C))\]

\noindent
with  

\[\lambda \eta\in \Omega^{1,0}(\Sigma,\Lambda_+\mathfrak{sl}(2,\C))\]

\noindent
such that  its residue at $\lambda=0$
 
 \noindent
\[\eta_{-1}=\text{Res}_{\lambda=0} (\eta)\]

\noindent
is a nowhere vanishing and nilpotent 1-form.\\

\noindent
A DPW potential $\eta$ gives rise to a loop of flat $\text{SL}(2,\C)$-connections. Let $\widetilde \Sigma$ denote the universal covering of $\Sigma$ and let

\[\Phi\colon \widetilde\Sigma\longrightarrow \Lambda\text{SL}(2,\C)\] 

\noindent
be the solution of the ODE 
\begin{equation}
\label{eq:ODE}d_\Sigma\Phi=\Phi \cdot \eta\end{equation}
\noindent
with initial value $\Phi(p) \in\Lambda\text{SL}(2,\C)$.  Then the Iwasawa decomposition $(F,B)$  of $\Phi$ gives smooth maps

\[F\colon\widetilde\Sigma\longrightarrow\Lambda\text{SU}(2) \quad \text{ and } \quad B\colon\widetilde\Sigma\longrightarrow\Lambda_+^\R\text{SL}(2,\C)\]

\noindent
and the associated family of flat connections of a minimal surface \cite{Bobenko_1991,hitchin} 

\[f\colon\tilde\Sigma\longrightarrow\S^3\]
is given by $\nabla^\lambda = d_\Sigma+(F^\lambda)^{-1}d_\Sigma F^\lambda$ satisfying

\begin{equation}\label{nablalambda}
d_\Sigma+F^{-1}d_\Sigma F=(d_\Sigma+\eta).B^{-1}=d_\Sigma+B \eta B^{-1}-d_\Sigma B B^{-1}.
\end{equation}
\noindent

\noindent
Identifying $\S^3\cong\text{SU}(2)$, the surface can therefore be reconstructed by the Sym-Bobenko formula 
\begin{equation}\label{symbob}f=F^{\lambda=1}(F^{\lambda=-1})^{-1}.\end{equation}\\

\noindent In this paper we are interested in constructing compact minimal surfaces with nontrivial topology. Thus we start with a DPW potential defined on such a Riemann surface $\Sigma$. The so-constructed minimal surface is well-defined on $\Sigma$ if its associated family of flat  connections $\nabla^\lambda$ satisfies the closing conditions (i)-(iii). For the corresponding DPW potential it is sufficient to have

\begin{enumerate}\label{rem:closing}
\item[(i)] $B$ has trivial monodromy, i.e., $B$ is well-defined on $\Sigma;$
\item[(ii)] the connections $d_\Sigma+\eta^{\lambda=\pm1}$ have trivial monodromy.
\end{enumerate}

\noindent
 Let $\gamma\in\pi_1(\Sigma,z_0)$ and let $\cal{M}(\Phi,\gamma)$ denotes the monodromy of $\Phi$ with respect to $\gamma.$ In terms of $\Phi$, the condition on the DPW potential is equivalent to:
\begin{equation}
\label{monodromy-problem}
\left\{\begin{array}{l}
\cal{M}(\Phi,\gamma)\in \Lambda SU(2)\\
\cal{M}(\Phi,\gamma)|_{\lambda=1}=\cal{M}(\Phi,\gamma)_{\lambda=-1}=\pm\Id_2\end{array}\right.\end{equation}
We refer to these conditions in \eqref{monodromy-problem} as the {\em Monodromy Problem}.

\subsubsection{Gauge freedom and apparent singularities}\label{subsubapp} $\; $\\ 
A DPW potential $\eta$ is not uniquely determined by its minimal immersion $f$. We rather have a gauge freedom. 
Consider a DPW potential $\eta$ on $\Sigma,$ and a holomorphic map 

$$\tilde B\colon \Sigma\longrightarrow\Lambda_+\SL(2,\C).$$ The gauged potential is defined to be

\begin{equation}\label{def:gauge}\tilde \eta = \eta.\tilde B:= \tilde B^{-1}\eta\tilde B+ \tilde B^{-1}d_\Sigma\tilde B .\end{equation}

\noindent
Due to the positivity of $\tilde B$, $\tilde \eta$ is again a DPW potential. Moreover,  
\[\tilde \Phi=\Phi \tilde B.\]

\noindent
is the unique solution of 

\[d_\Sigma\tilde\Phi=\tilde\Phi\tilde \eta \quad  \text{ with initial condition } \quad \tilde\Phi(p)= \Phi(p) \tilde B(p).\]

\noindent
 Let $F_0B_0=B^{\lambda=0}\tilde B^{\lambda=0}$ be the finite dimensional Iwasawa decomposition into a unitary and an upper triangular matrix with positive diagonal entries. Then

\[\tilde\Phi=\left (FF_0\right)\left(B_0  (F_0B_0)^{-1} B\tilde B\right)\]
and 

\[\tilde F = F F_0 \quad \quad  \tilde B =  B_0(F_0B_0)^{-1} B\tilde B\]
is the Iwasawa decomposition of $\tilde\Phi.$ Therefore, the two DPW potentials $\eta $ and $\tilde \eta$ yield the same minimal immersion $f$ via Sym-Bobenko formula \eqref{symbob} . \\

\noindent
In particular, certain singularities of $\eta$ can be removed using the gauge freedom. Let $\eta$ be a meromorphic potential with a singularity at
$q\in\Sigma$ .  If there exists a positive gauge  $\tilde B\colon\Sigma\setminus\{q\}\to\Lambda_+\SL(2,\C)$ such that
$\tilde B.\eta$ extends holomorphically to $q, $ then the surface obtained from $\eta$ extends real analytically to $q.$
Singularities of this type are called {\it{apparent} singularities}.
\begin{remark}
\label{remark:apparent-singularities}
In order to obtain a compact minimal surface $f\colon \Sigma\to\S^3$, its DPW potential necessarily has apparent singularities.
This follows for instance from the area formula in Corollary \ref{cor:areares}.
\end{remark}

\begin{remark}
The DPW method can be generalised to potentials $\eta^\lambda$ that are only defined for $\lambda \subset D_r = \{\lambda \in \C^* \mid |\lambda| \leq r\}$ for $r \in (0, 1]$. Details and proofs can be found in \cite{SKKR} and \cite{KRS}. 
\end{remark}

\section{A DPW potential for Lawson surfaces of high genus}\label{sec:tech}

\noindent
To choose our potential we take advantage of the symmetries of Lawson surface $\xi_{m,k}$.
We assume the potential to be invariant under the $\Z_{k+1}$ action as in \cite{He}.
%Let $\nabla^\lambda$ be its associated family of flat connections and $\phi: \Sigma \rightarrow \Sigma$ denote the $\Z_{k+1}$-symmetry. Then the gauge class of $\nabla^\lambda$ is invariant under pull-back by $\phi,$ i.e., $[\nabla^\lambda]= [\phi^*\nabla^\lambda]$, by construction. As in the case $(m,k)=(1,2)$ \cite{He}, we expect that there exist a family of symmetric meromorphic connections $\wt \nabla^\lambda$ with singularities at the $2m+2$ fix points of $\phi$. 
%These new connections $\wt \nabla^\lambda$ gives rise to well-defined meromorphic connections on the quotient $\Sigma/\Z_{k+1} \cong \C P^1$. It turns out to be convenient to use a DPW potential on $\C P^1$ with an additional apparent singularity at $z=\infty$.\\
In this section we show the existence of DPW potentials on a $(2m+3)$-punctured sphere 
$$\C\setminus \{p_0, ..., p_{2m+1}\}$$
with an apparent singularity at $ z = \infty$ such that the Monodromy Problem \eqref{monodromy-problem} is solved on a finite cover $\Sigma$ of the punctured sphere, branched at $p_j, j \in \{0, ...2m+1\}$. 
This gives rise to countably infinite many compact and embedded minimal surfaces in $\S^3.$ In section \ref{sec:construction} we show that these minimal surfaces coincide with the Lawson surfaces $\xi_{m,k}$ for $k \gg1.$
%The key idea is to relax the closing conditions in order to deform the DPW potential of a point with blow-up limit being a well-known surface -- here the saddle tower -- and thereby desingularize and change its topology. Therefore, it is key to understand convergence properties for DPW potentials. We start with fixing notations. \\

\subsection{Notations for functional spaces}\label{section-notations}$\;$\\
We follow the notations set in \cite{nnoids}: 
For  $f\in L^2(\S^1,\mathbb C)$ consider its Fourier series
\[f=\sum_k f_k \lambda^k.\] For $\rho>1$ define

\[\parallel f\parallel_\rho=\sum |f_k| \rho^{|k|}\leq\infty\]
and let

 \[\mathcal W_\rho:=\{f\in L^2\mid \parallel f\parallel_\rho<\infty\}\] 
be the set of Fourier series absolutely convergent  on the annulus 

$$\A_{\rho}=\{\lambda\in\C\mid \tfrac{1}{\rho} < |\lambda| <\rho\}.$$
\begin{remark}
The notation is also used for arbitrary loop spaces $\mathcal H$:
$\mathcal H_\rho$ denotes the subspace of $\mathcal H$ of loops whose entries
are in $\mathcal W_\rho$.
Then $\Lambda SL(2,\C)_{\rho}$, $\Lambda SU(2)_{\rho}$ and $\Lambda_+^{\R} SL(2,\C)_{\rho}$
are Banach Lie groups and Iwasawa decomposition is a smooth diffeomorphism from
$\Lambda SL(2,\C)_{\rho}$ to $\Lambda SU(2)_{\rho}\times\Lambda_+^{\R} SL(2,\C)_{\rho}$
(see  Theorem 5 in \cite{minoids}).
 \end{remark}
\noindent
Moreover, let

\[\cal{W}^{\geq 0}_\rho:=\{f=\sum_k f_k \lambda^k\in \mathcal W_\rho\mid f_k=0 \;\forall \; k<0\}\]
denote the space of those  loops $f\in L^2(\S^1,\mathbb C)$ that can be extended to a holomorphic function on the unit disc.  Similarly, let
\[\cal{W}^{>0}_\rho:=\{f=\sum_k f_k \lambda^k\in \mathcal W_\rho\mid f_k=0 \;\forall \; k\leq0\}\]
\[\cal{W}^{<0}_\rho:=\{f=\sum_k f_k \lambda^k\in \mathcal W_\rho\mid f_k=0 \;\forall \; k\geq0\}\]
denote the positive and negative space, respectively. Therefore we can decompose every $f\in \mathcal W_\rho$
$$f=f^+ + f^0+ f^-$$
into its positive and negative component $f^\pm \in W^{\gtrless 0}_\rho$, and a constant component $f^0 = f_0$. \\

\noindent
On $\mathcal W_\rho$ there exists two important involutions. The first is

  \[^*\colon \mathcal W_\rho\longrightarrow\mathcal W_\rho; f \longmapsto f^*,\]
  
\noindent 
where $f^*$ is determined by
 \[f^*(\lambda)=\overline{f\left(\frac{1}{\overline{\lambda}}\right)} \quad \text{ for } \lambda \in \A_{\rho}.\]

\noindent
The second involution is the conjugation of $f \in \mathcal W_\rho$ defined by:

\begin{equation}
\label{eq:conjugate}\overline{f}(\lambda)=\overline{f(\overline{\lambda})}.\end{equation}

\noindent
Let $\cal{W}_{\R}$, $\cal{W}^{\geq 0}_{\R}$ etc. denote real subspaces of $\cal{W}_\rho$, $\cal{W}^{\geq 0}_{\rho}$ 
satisfying $\overline{f}=f$. Functions in $\cal{W}_{\R}$ can be decomposed as
$f(\lambda)=\sum f_k \lambda^k$ with real coefficients $f_k\in\R$.
Observe that conjugation and star commute:
$$\overline{u^*}(\lambda)=\overline{u}^*(\lambda)=u\left(\frac{1}{\lambda}\right).$$
\begin{remark}
The notations for the decomposition of $\mathcal W_\rho$ into $\mathcal W ^{\gtrless 0}$ etc, the involutions and the real subspaces carry over to loop spaces $\cal H$.\end{remark}

 \subsection{Convergence to a Saddle Tower}
\label{sec:scherk}$\;$

\noindent
The Lawson surfaces $\xi_{m,k}$ converge for $m$ fixed and $k \rightarrow \infty$ to the union
of m + 1 great spheres intersecting in a great circle.
Moreover, the blow-up of $\xi_{m,k}$ %on the limit intersection circle converges as 
converges for $k\to\infty$
to a saddle tower with $2m+2$ ends.
The following blow-up result is adapted from Theorem 4 in \cite{minoids}. Though written for CMC surfaces in $\R^3$, an analogue statement also holds for the ambient space $\S^3$.
We omit its proof, as we will only use it as a heuristic to construct our potential for Lawson surfaces.

\begin{theorem}\label{blowuplimit}
Let $\Sigma$ be a Riemann surface, $\epsilon >0$ and $I = (-\epsilon, \epsilon) \subset \R.$
Moreover, let $(\eta_t)_{t\in I}$ a family of DPW potentials on
$\Sigma$ and $(\Phi_t)_{t\in I}$ the corresponding family of solutions.
Fix a base point $z_0\in\widetilde{\Sigma}$ and assume 
\begin{enumerate}
\item $(t,z)\mapsto \eta_t(z,\cdot)$ and $t\mapsto \Phi_t(z_0,\cdot)$ are $C^1$ maps into
$\left(\Lambda\sl(2,\C)\right)_{\rho}$ and $\left(\Lambda SL(2,\C)\right)_{\rho}$, respectively.
\item $\Phi_t$ solves the Monodromy Problem \eqref{monodromy-problem} for all $t\in I$.
\item $\Phi_0(z,\lambda)$ is independent of $\lambda$:
$$\Phi_0(z,\lambda)=\matrix{\alpha(z)&\beta(z)\\\gamma(z)&\delta(z)}.$$
\end{enumerate}
Let $f_t
:\Sigma\to\S^3\cong SU(2)$ be the corresponding family of minimal immersions via DPW. (Since $F_0(z)$ is independent of $\lambda$, $f_0\equiv \Id$.)
Then 
$$\psi \colon \Sigma\longrightarrow T_{\Id} SU(2) \cong  \R^3,  \quad \psi(z):= \lim_{t\rightarrow 0}\frac{1}{t}\left(f_t(z)-\Id\right)$$ 
 is a well-defined and (possibly branched) minimal immersion with the following Weierstrass data (with "vertical" axis $x_2$ and "horizontal" axes $x_3,x_4$ in the tangent plane $x_1=1$ of $\S^3$ at $\Id$):
$$g(z)=\frac{\ii \alpha(z)}{\gamma(z)}\quad\mbox{ and }\quad
\omega=-4
\gamma(z)^2\Res_{\lambda}\left (\frac{\partial \eta_{t;12}}{\partial t}|_{t=0} \right),$$
where $\eta_{t; 12}$ is the upper right entry of the $2\times 2$ potential $\eta_t$ and the residue taken with respect to its expansion in $\lambda.$ 
The convergence is hereby uniform $C^1$on compact subsets of $\Sigma$.
\end{theorem}

\noindent 
We aim at finding a family of DPW potentials $\eta_t,$ $t \sim 0,$
with a saddle tower (see \ref{sec:karcher}) as its blow-up limit $\psi$ at $t=0$. The Gauss map $g$ of the saddle tower \eqref{saddletower} suggests  to choose
$$\eta_0= \matrix{0&0\\mz^{m-1}dz&0}.$$
The corresponding solution with initial value $\Phi_0(z=0) = \Id$ is then given by 
$$\Phi_0(z)=\matrix{1&0\\z^m&1}$$ 
which is independent of $\lambda$ and yields the correct $g$ according to Theorem \ref{blowuplimit}.
The meromorphic 1-form $\omega$ of the saddle tower suggests that $\eta_t$ should have simple poles
with residue of order $t$ at the $2m+2$ roots of unity.
\subsection{The potential}$\;$\\
Let $m \in \N^*$ be fixed and define $n=2m+2$.  We consider the ansatz
$$\eta_{t}=\matrix{0&0\\mrz^{m-1}dz&0}+t\sum_{j=0}^{n-1} A_j(\lambda)\frac{dz}{z-p_j},$$
where $A_i \in \left(\Lambda\mathfrak{sl}(2, \C)\right)_\rho$ and the initial condition
$$\Phi_t(z=0)=\Id.$$
Here $r$ and $t$ are real parameters with $r \in (1-\epsilon, 1+\epsilon)$ and $t \in (-\epsilon, \epsilon)$ for some $\epsilon > 0,$  and
$$p_j=e^{2\pi\ii j/n}\quad\mbox{ for $0\leq j\leq 2m+1$}.$$ The parameter $r$ will be later determined by solving the Monodromy Problem. Its initial value at $t=0$ is $r=1$. \\
\subsection{Symmetries}\label{sec:symm}$\;$ \\
Due to the symmetries of the Lawson surfaces $\xi_{m,k}$, we also assume the potentials $\eta_{t}$ to be symmetric.
Let
$$A_0(\lambda)=\matrix{a(\lambda)&\lambda^{-1}b(\lambda)\\
\lambda c(\lambda)&-a(\lambda)}$$
with  functions $a$, $b$, $c$ in $\cal{W}^{\geq 0}_{\R}$.  We assume
$$A_{j+1}(\lambda)=D^{-1}A_j(-\lambda)D\quad\mbox{ for $0\leq j\leq n-2$ with }$$
$$D=\matrix{e^{\ii\pi m/n}&0\\0&e^{-\ii\pi m/n}}.$$
\noindent
Observe that
$$e^{2\ii\pi m/n}=e^{i\pi(n-2)/n}=-e^{-2\ii\pi/n}.$$
Hence writing $A_j=\matrix{a_j&\lambda^{-1}b_j\\\lambda c_j&-a_j}$ with  functions $a_j$, $b_j$, $c_j$ in $\cal{W}^{\geq 0}_{\R}$ we obtain
$$\left\{\begin{array}{l}
a_{j+1}(\lambda)=a_j(-\lambda)\\
b_{j+1}(\lambda)=e^{2\ii\pi/n}b_j(-\lambda)\\
c_{j+1}(\lambda)=e^{-2\ii\pi/n}c_j(-\lambda)
\end{array}\right.$$
and
\begin{equation}\label{explicit-potential}
\begin{split}
\sum_{j=0}^{n-1}\frac{a_j}{z-p_j}&=\frac{nz^m}{2}\left(\frac{a(\lambda)}{z^{m+1}-1}+\frac{a(-\lambda)}{z^{m+1}+1}\right)\\
\sum_{j=0}^{n-1}\frac{b_j}{z-p_j}&=\frac{n}{2}\left(\frac{b(\lambda)}{z^{m+1}-1}-\frac{b(-\lambda)}{z^{m+1}+1}\right)\\
\sum_{j=0}^{n-1}\frac{c_j}{z-p_i}&=\frac{nz^{m-1}}{2}\left(\frac{c(\lambda)}{z^{m+1}-1}+\frac{c(-\lambda)}{z^{m+1}+1}\right).
\end{split}
\end{equation}

\noindent
The symmetries of $\eta_t$ are induced by
$\delta(z)=e^{2\pi\ii/n} z$ and
$\sigma(z)=\overline{z}$. We have
$$\delta^*\eta_t(z,\lambda)=D^{-1}\eta_t(z,-\lambda)D,\qquad
\delta^*\Phi_t(z,\lambda)=D^{-1}\Phi_t(z,-\lambda)D$$
and
$$\sigma^*\overline{\eta_t}=\eta_t,\qquad\sigma^*\overline{\Phi_t}=\Phi_t$$
which, remembering notation \eqref{eq:conjugate}, means
$$\sigma^*\overline{\eta_t(\cdot,\overline{\lambda})}=\eta_t(\cdot,\lambda)$$
and likewise for $\Phi_t$.

\noindent
\subsection{The Monodromy Problem}$\;$\\
Let $\gamma_0,\cdots,\gamma_{n-1}$ be generators of the fundamental group
$\pi_1(\C\setminus\{p_0,\cdots,p_{n-1}\},0)$, with
$\gamma_j$ enclosing the singularity $p_j$. Let
$M_j(t)=\cal{M}(\Phi_t,\gamma_j)$ be the monodromy of $\Phi_t$ along $\gamma_j$.
We want to solve the following problem for all $j$:
\begin{equation}
\label{monodromy-problem2}
\left\{\begin{array}{l}
M_j(t)\in\Lambda SU(2)\\
M_j(t)|_{\lambda=\pm 1}\mbox{ diagonal}\\
M_j(t) \mbox{ has eigenvalues $e^{\pm 2\pi\ii t}$.}
\end{array}\right.
\end{equation}
We will see in Section \ref{sec:con} that provided Problem \eqref{monodromy-problem2} is solved,
taking $t=\frac{1}{2(k+1)}$,
the potential $\eta_t$ pulls back on a $(k+1)$-branched cover to a potential with apparent singularities
solving the Monodromy Problem \eqref{monodromy-problem}. This yields the desired closed minimal surface.\\

\noindent
Regarding symmetries, we have since $\delta(\gamma_j)=\gamma_{j+1}$:
\begin{equation}
\label{eq-recursive-M}
M_{j+1}(t)(\lambda)=D^{-1}M_j(t)(-\lambda)D.
\end{equation}
Because $D$ is unitary and diagonal, it suffices to solve Problem \eqref{monodromy-problem2} for
$j=0$. From now on, we write $M=M_0$.
Since $\sigma(\gamma_0)=\gamma_0^{-1}$, we also have, using the notation \eqref{eq:conjugate}\begin{equation}\label{eq:realMono}
 M(t)=\left(\overline{M(t)}\right)^{-1}.
\end{equation}

\begin{remark} It will turn out that provided Problem \eqref{monodromy-problem2} is solved,
the singularity at $z=\infty$ is apparent, see
Section \ref{sec:reginf}.
\end{remark}
\noindent
At $t=0$ the solution of \begin{equation}\label{eqn-sol}d_\Sigma\Phi_t=\Phi_t\eta_t \quad \text{and}\quad \Phi_t(0)=\Id\end{equation}
is given by 
$$\Phi_0(z)=\matrix{1&0\\r z^m&1}$$
with trivial monodromy, i.e.,
 $$M(0)=\Id.$$
 
\noindent
Hence
$$\wt{M}(t):=\frac{1}{t}\log M(t)$$
extends smoothly at $t=0$, with
$\wt{M}(0)=M'(0)$.
When $t\neq 0$, Problem \eqref{monodromy-problem2} is equivalent to
\begin{equation}
\label{monodromy-problem3}
\left\{\begin{array}{l}
\wt{M}(t)\in\Lambda \su(2)\\
\wt{M}(t)|_{\lambda=\pm 1}\mbox{ diagonal}\\
\wt{M}(t) \mbox{ has eigenvalues $\pm 2\pi\ii$.}
\end{array}\right.
\end{equation}
From the symmetry \eqref{eq:realMono} of $M$ we deduce
$$\wt{M}(t)=-\overline{\wt{M}(t)}.$$
Following \cite{nnoids} we compute 
\begin{eqnarray*}
\wt{M}(0)&=&\int_{\gamma_0}\Phi_0\frac{\partial \eta_t}{\partial t}|_{t=0}\Phi_0^{-1}
=2\pi\ii\Res_{p_0}\left [\matrix{1&0\\rz^m &1}A_0\matrix{1&0\\-rz^m&1}\frac{dz}{z-p_0} \right]\\
&=&2\pi\ii\matrix{a-\lambda^{-1}rb&\lambda^{-1}b\\ 2ra-\lambda^{-1}r^2b+\lambda c&-a+\lambda^{-1}rb}.\end{eqnarray*}
Let $\xx=(r,a,b,c)$ denote the vector of parameters. To highlight the parameters, we denote the potential determined by $\xx$ as
\[\eta_t=\eta_t^\xx.\]
The initial value of $\xx$, denoted by $\xx_0$, is taken to be
\begin{equation}
\label{eq-central}
r=1,\quad a=\lambda,\quad b=\frac{\lambda^2-1}{2}\quad\mbox{ and }\quad
c=-2.\end{equation}
For these values of the parameters, we obtain at $t=0$
\begin{equation}
\label{eq-wtM0}
\wt{M}(0)=\pi\ii\matrix{\lambda+\lambda^{-1}&\lambda-\lambda^{-1}\\\lambda^{-1}-\lambda&-\lambda-\lambda^{-1}}
\end{equation}
so  Problem \eqref{monodromy-problem3} is solved at $t=0$.
\begin{remark}
Assuming $r=1$, one can prove as in \cite{nnoids} that \eqref{eq-central} is the only solution to
Problem \eqref{monodromy-problem3} at $t=0$, up to $(a,b,c)\mapsto (-a,-b,-c)$.
\end{remark}
\subsection{Solving the Monodromy Problem  for $t\neq 0$}\label{sec:implicit} $\;$\\
For a parameter $(t,\xx)$ and the corresponding solution $\Phi_t$ of \eqref{eqn-sol} and its monodromy $M(t)\in \Lambda\SL(2, \C)$ 
we define
$$\cal{F}(t,\xx)=\frac{1}{2\pi\ii}(\wt{M}_{11}(t)+\wt{M}_{11}(t)^*)$$
$$\cal{G}(t,\xx)=\frac{1}{2\pi\ii}(\wt{M}_{21}(t)+\wt{M}_{12}(t)^*)$$
$$\cal{H}_1(t,\xx)=\frac{1}{2\pi\ii}\wt{M}_{12}(t)(\lambda=1)$$
$$\cal{H}_2(t,\xx)=\frac{1}{2\pi\ii}\wt{M}_{12}(t)(\lambda =-1)$$
$$\cal{K}(\xx)=-\det(A_0)|(\lambda=0)= (a^0)^2+b^0c^0.$$

\begin{proposition}
Problem \eqref{monodromy-problem3} is equivalent to
\begin{equation}\label{IFTeqn}
\left\{\begin{array}{l}
\cal{F}(t,\xx)=0\\
\cal{G}(t,\xx)=0\\
\cal{H}_1(t,\xx)=0\\
\cal{H}_2(t,\xx)=0\\
\cal{K}(\xx)=1\,.\end{array}\right.\end{equation}
\end{proposition}
\begin{proof} The first four equations of \eqref{IFTeqn} are clearly equivalent to the first two equations of
\eqref{monodromy-problem3}.
Regarding the last one, by standard fuchsian system theorem, $\wt{M}(t)$ has the same eigenvalues as
$2\pi\ii A_0$, so the last equation of \eqref{monodromy-problem3} is equivalent to
$\det(A_0)=-1$. Provided $\wt{M}(t)\in\Lambda\su(2)$, its eigenvalues are imaginary for $\lambda$ on the unit circle, so
the eigenvalues of $A_0$ are real. Now $\det(A_0)$ is a homorphic function of $\lambda$ in the unit disk
which is real on the unit circle, so it must be constant.
\end{proof}

\noindent
From the symmetries we have
$$\overline{\cal{F}}=-\frac{1}{2\pi\ii}(\overline{\wt{M}_{11}}+\overline{\wt{M}_{11}^*})
=-\frac{1}{2\pi\ii}(-\wt{M}_{11}-\wt{M}_{11}^*)=\cal{F}.$$
Hence $\cal{F}(t,\xx)\in\cal{W}_{\R}$. In the same way, $\cal{G}(t,\xx)\in\cal{W}_{\R}$.
Further, since $\cal{F}^*=-\cal{F}$ by definition, we obtain $\cal{F}^+(\lambda)=-\cal{F}^-(\tfrac{1}{\lambda})$, and therefore we do not have to
solve $\cal{F}^-=0$ separately. Moreover,
$$\cal{F}\in\cal{W}_{\R}\Rightarrow \cal{F}^0\in\R$$
$$\overline{\cal{F}^0}=(\cal{F}^*)^{0}=-\cal{F}^0\Rightarrow \cal{F}^0\in\ii\R.$$
Hence $\cal{F}^0(t,\xx)=0$ automatically holds by symmetry.\\

\noindent
Differentiating $\cal F$ and $\cal G$ with respect to $\xx$ at $(0,\xx_0)$ given by \eqref{eq-central} gives
\begin{equation}\label{dFdG}
\begin{split}
d\cal{F}&=da-\lambda^{-1}db-da^*+\lambda db^*+(\lambda^{-1}-\lambda) dr\\
d\cal{G}&=2da-\lambda^{-1}db+\lambda dc-\lambda db^*+(\lambda+\lambda^{-1})dr.
\end{split}
\end{equation}
Write
$$b(\lambda)=b^0+\lambda \wt{b}(\lambda)\quad\mbox{ with } \wt{b}\in\cal{W}^{\geq 0}_{\R}.$$
Then (recalling that $b^0\in\R$)

\begin{equation}\label{dFdG+}
\begin{split}
d\cal{F}^+&=da^+ -d\wt{b}^+ +\lambda db^0 -\lambda dr\\
d\cal{G}^+&=2da^+ -d\wt{b}^++\lambda dc-\lambda db^0+\lambda dr\\
d\cal{G}^-&=-\lambda^{-1}db^0-(d\wt{b}^+)^*+\lambda^{-1}dr\\
(d\cal{G}^-)^*&=-\lambda db^0-d\wt{b}^++\lambda dr.
\end{split}
\end{equation}
The Jacobian of $(\cal{F}^+,\cal{G}^+,(\cal{G}^-)^*)$ with respect to
$(a^+,\wt{b}^+,\lambda c)$ is
$$\left(\begin{array}{ccc}1&-1&0\\2&-1&1\\0&-1&0\end{array}\right)$$
so this operator is an automorphism of $(\cal{W}^{>0}_{\R})^3$.
(Both variables and functions are in $\cal{W}^{>0}_{\R}$ by definition and the previous symmetry arguments.)
Therefore, applying the Implicit Function Theorem, the equations $\cal{F}^+=0$,
$\cal{G}^+=0$ and $\cal{G}^-=0$ uniquely determine the parameters
$a^+$, $\wt{b}^+$ and $c$ as functions of $t$ and the remaining parameters
$r,a^0,b^0,\wt{b}^0$.\\

\noindent
It remains to solve four real equations $\cal{G}^0=0$, $\cal{H}_1=0$, $\cal{H}_2=0$
and $\cal{K}=0$
with the remaining  four parameters
$(r,a^0,b^0,\wt{b}^0)\in\R^4$. The derivatives of the functions $a^+$, $\wt{b}^+$
and $c$ with respect to these parameters satisfy
\begin{equation*}
\begin{split}
d\wt{b}^+&=-\lambda db^0+\lambda dr,\\
da^+&=-2\lambda db^0+2\lambda dr,\\ 
dc& =4db^0-4dr,
\end{split}
\end{equation*}
which is obtained by inserting  
$d\cal{F}^+=0$, $d\cal{G}^+=0$
and $d\cal{G}^-=0$ into \eqref{dFdG+}.
\noindent
With these we obtain
\begin{equation*}
\begin{split}
d\cal{G}^0&=2da^0-2d\wt{b}^0\\
d\cal{H}_1&=db(1)=db^0+d\wt{b}^0+d\wt{b}^+(1)=db^0+d\wt{b}^0-db^0+dr=d\wt{b}^0+dr\\
d\cal{H}_2&=-db(-1)=-db^0+d\wt{b}^0+d\wt{b}^+(-1)=-db^0+d\wt{b}^0+db^0-dr=d\wt{b}^0-dr\\
d\cal{K}&=-\frac{1}{2}dc^0-2db^0=4db^0-4dr.
\end{split}
\end{equation*}
The Jacobian of $(\cal{G}^0,\cal{H}_1,\cal{H}_2,\cal{K})$ with respect to $(a^0,b^0,\wt{b}^0,r)$
is an automorphism of $\R^4.$ Therefore,   these equations uniquely determine the remaining parameters
$(a^0,b^0,\wt{b}^0,r)$ as smooth functions for $t\sim 0$ by Implicit Function Theorem.
So we have proven the following proposition:
\begin{proposition}\label{prop:IFT}
For $t>0$ small, there exists a unique $\xx(t)$ in a neighbourhood of $\xx_0$ such that \eqref{IFTeqn} holds. In other words, the DPW potential $\eta^{\xx(t)}_t$ solves the Monodromy Problem \eqref{monodromy-problem2}.
\end{proposition}
\noindent
We shall need the value of the monodromies at $\lambda=\pm 1$:
\begin{proposition}\label{prop-monodromy-at-1}
The monodromy of the solution of \eqref{eqn-sol} for  $\eta_t^\xx$ determined by Proposition \ref{prop:IFT} satisfies
$$M_j(t)(\lambda=1)=\matrix{e^{2\pi\ii t}&0\\0&e^{-2\pi\ii t}}^{(-1)^j}\quad\mbox{ and }\quad
M_j(t)(\lambda=-1)=\matrix{e^{-2\pi\ii t}&0\\0&e^{2\pi\ii t}}^{(-1)^j}$$
for $0\leq j\leq n-1$. Moreover, $\det(A_j(t))(\lambda)=-1$ for all $\lambda\in\S^1$.
\end{proposition}
\noindent
\begin{proof}
By equation \eqref{monodromy-problem3}, $\wt{M}(t)(\pm 1)$ is diagonal with eigenvalues $\pm 2\pi\ii$. From Equation \eqref{eq-wtM0} at $t=0$, we obtain by continuity
$$\wt{M}(t)(\lambda=1)=2\pi\ii\matrix{1&0\\0&-1} \quad\mbox{ and }\quad
\wt{M}(t)(\lambda=-1)=2\pi\ii\matrix{-1&0\\0&1}.$$
The proposition follows from $M(t)=\exp(t\wt{M}(t))$ and Equation \eqref{eq-recursive-M}.\end{proof}

\subsection{Regularity at $z=\infty$}\label{sec:reginf}
The following proposition guarantees that the surfaces constructed by $\eta^\xx_t$ in Proposition \ref{prop:IFT} extends smoothly to $z=\infty$, see also Subsection \ref{subsubapp}.
\begin{proposition}\label{apparent}
For $t \sim 0$ let  $\eta_t=\eta^{\xx(t)}_t$ be the unique solution of the Monodromy Problem with parameters $\xx(t).$ Then $z=\infty$ is an apparent singularity of $\eta_t$.
\end{proposition}
\noindent
\begin{proof} Let $\gamma_{\infty}=\prod_{i=0}^{n-1}\gamma_i$ and $M_{\infty}(t)$ be the monodromy of $\Phi_t$ corresponding
to $\gamma_{\infty}$.
By Proposition \ref{prop-monodromy-at-1}, $\Phi_t$ solves the following Monodromy Problem:

$$\left\{\begin{array}{l}
M_{\infty}(t)\in\Lambda SU(2)\\
M_{\infty}(t)(\pm 1)=\Id\,.
\end{array}\right. $$
Consider the gauge

$$G_0(z)=\matrix{z^{-m}&\frac{-1}{r}\\0&z^m}.$$
Then

$$\eta_0. G_0=\matrix{0&0\\m z^{-m-1}dz&0}$$
which is holomorphic at $\infty$ since $m\geq 1$.
We introduce a parameter $s\in\cal{W}^{\geq 0}_\rho$ and define

$$G_s(z,\lambda)=\matrix{z^{-m}&\frac{1}{r}(s(\lambda)-1)\\0&z^m}.$$
Let $\wh{\eta}_t=\eta_t. G_s$.
A computation reveals that

$$\wh{\eta}_{t;21}=\frac{mr\,dz}{z^{m+1}}+t\sum_{i=0}^{n-1}
\frac{\lambda c_i\,dz}{z^{2m}(z-p_i)}$$
is holomorphic at $z=\infty$ and

$$\wh{\eta}_{t;11}=\frac{-ms\,dz}{z}+t\sum_{i=0}^{n-1}\left[ \frac{a_i\,dz}{z-p_i}+\frac{\lambda c_i(1-s)dz}{rz^m(z-p_i)}\right],$$
which is holomorphic at $z = \infty$ by choosing

$$s=\frac{t}{m}\sum_{i=0}^{n-1} a_i.$$
Finally,

$$\wh{\eta}_{t;12}=\frac{s}{r}(1-s)m z^{m-1}dz+t\sum_{i=0}^{n-1}\left[\frac{2a_i(s-1)z^m\,dz}{r(z-p_i)}+
\frac{b_i z^{2m}\,dz}{\lambda(z-p_i)}-\frac{\lambda c_i(s-1)^2dz}{r^2(z-p_i)}\right].$$
We use $w=\frac{1}{z}$ as a local coordinate. From Equation \eqref{explicit-potential}, we obtain

$$\wh{\eta}_{t;12}=\lambda^{-1}B(\lambda)\frac{dw}{w^{m+1}}+O(w^0\,dw)$$
with

$$B(\lambda)=-\lambda r^{-1}s(1-s)m-\lambda r^{-1}t(s-1)n(a(\lambda)+a(-\lambda))
-t(b(\lambda)-b(-\lambda)).$$
In particular, $B(0)=0$.
By Theorem \ref{thm:appendix} in the Appendix \ref{app}, $\wh{\eta}_t$ is holomorphic at $z=\infty$.
(Note that $\wh{\eta}_{0;21}=-m w^{m-1}dw$ so to apply Theorem \ref{thm:appendix}, we make the change of coordinate $v=kw$ with $k^m=-1$.)\end{proof}

\begin{remark}
\label{remark-imm-infty}
The coefficient $\wh{\beta}$ of $\lambda^{-1}$ in $\wh{\eta}_{t;12}$ is obtained from the
coefficient of $\lambda^{-1}$ in $\eta_{t;12}$ multiplied by $(G_{s;22})^2$.
Equation \eqref{explicit-potential} then gives
$$\wh{\beta}=\frac{ n b^0 z^{n-2}dz}{z^n-1}$$
which does not vanish at $\infty$. Hence the immersion obtained from the DPW method
will be unbranched at $\infty$.
\end{remark}
\section{Derivatives of the parameters}\label{sec:der}
\noindent
In this section, we consider the unique family $\eta^{\xx(t)}_t$ from Proposition \ref{prop:IFT}   solving the Monodromy Problem.  Let $\xx(t)=( r(t), a(t), b(t), c(t))$. We want to compute the time derivatives of the parameters. 
\subsection{Time parity of the potential}$\;$\\
The following proposition facilitates the computations of the derivatives of the parameters.
\begin{proposition}
\label{prop-time-parity}
Assume that $\eta_t=\eta_t^{\xx(t)}$ is the unique family from Proposition \ref{prop:IFT}.
Then

$$\eta_{-t}(z,-\lambda)=\eta_t(z,\lambda).$$
This is equivalent to

$$\left\{\begin{array}{l}
a(-t)(-\lambda)=-a(t)(\lambda)\\
b(-t)(-\lambda)=b(t)(\lambda)\\
c(-t)(-\lambda)=c(t)(\lambda)\\
r(-t)=r(t)\,.\end{array}\right.$$
In particular, $\frac{d^k}{dt^k}a(t=0)$ is an odd function of $\lambda$ and vanishes at $\lambda=0$, for all even $k$.
\end{proposition}
\noindent
\begin{proof} Let

$$\wh{\eta}_t(z,\lambda)=\eta_{-t}(z,-\lambda).$$
Then

$$\wh{\eta}_t(z,\lambda)=\matrix{0&0\\m\wh{r}(t)z^{m-1}dz&0}+t\sum_{j=0}^{n-1}
\wh{A}_j(t)(\lambda)\frac{dz}{z-p_j}$$
with

$$\wh{r}(t)=r(-t)\quad\mbox{ and } \quad
\wh{A}_j(t)(\lambda)=-A_j(-t)(-\lambda).$$
Hence

$$\wh{A}_0(t)=\matrix{\wh{a}(t)&\lambda^{-1}\wh{b}(t)\\\lambda\wh{c}(t)&-\wh{a}(t)}
\quad\mbox{ with }\quad \left\{\begin{array}{l}
\wh{a}(t)(\lambda)=-a(-t)(-\lambda)\\
\wh{b}(t)(\lambda)=b(-t)(-\lambda)\\
\wh{c}(t)(\lambda)=c(-t)(-\lambda)\,.\end{array}\right.$$

\noindent
Let $\wh{\xx}(t)=(\wh{a}(t),\wh{b}(t),\wh{c}(t),\wh{r}(t)).$
Observe that at $t=0$, $\wh{\xx}(0)=\xx(0)$.
Let $\wh{\Phi}_t$ be the solution of $d\wh{\Phi}_t=\wh{\Phi}_t\wh{\eta}_t$ with initial condition
$\wh{\Phi}_t(0)=\Id$. Then $\wh{\Phi}_t(z,\lambda)=\Phi_{-t}(z,-\lambda)$.
Hence $\wh{\Phi}_t$ solves Equation \eqref{IFTeqn}.
By uniqueness in the Implicit Function Theorem, $\wh{\xx}(t)=\xx(t)$ for all $t$ in a neighbourhood of $0$.\end{proof}
\subsection{First order derivatives}
\begin{proposition}
\label{prop-derivatives}
The $t$-derivatives of the parameters $\xx(t)$ solving \eqref{IFTeqn} at $t=0$ are given by
\begin{equation}
\label{eq-derivatives}
a'(0)=(1-\lambda^2)\kappa_m,\quad
b'(0)=(\lambda-\lambda^3)\kappa_m,\quad
c'(0)=0,\quad
r'(0)=0.\end{equation}
where
\begin{equation}\label{def:kappa}\kappa_m=\frac{n}{2}\int_{0}^1\frac{(1-x^m)^2}{1-x^n}dx,\quad n=2m+2.\end{equation}
\end{proposition}
\noindent
The values of $\kappa_m$ for small values of $m$ are tabulated below:
\begin{center}
\begin{tabular}{|c|c|}
\hline
$m$&$\kappa_m$\\
\hline
1&$\ln 2$\\
2&$\frac{3}{2}\ln 3$\\
3&$2\ln 2+\sqrt{2}\ln(1+\sqrt{2})$\\
4&$\frac{5}{4}\ln 5+\frac{\sqrt{5}}{2}\ln(2+\sqrt{5})$\\
5&$\ln 2+\frac{3}{2}\ln 3+\sqrt{3}\ln(2+\sqrt{3})$\\
\hline\end{tabular}
\end{center}

\noindent
Proof of Proposition \ref{prop-derivatives}:
First of all, by Proposition \ref{prop-time-parity}, $r'(0)=0$. Let
$$N(t):=M'(t)M(t)^{-1}.$$
Since the Monodromy Problem is solved for $\eta^{\xx(t)}_t$, we have by Proposition \ref{prop-monodromy-at-1}:

$$\left\{\begin{array}{l}
N(t)\in\Lambda\su(2)\\
N(t)(\lambda=\pm 1)=\pm 2\pi\ii \matrix{1&0\\0&-1}\end{array}\right.$$
from which we deduce

\begin{equation}
\label{problem-dN}
\left\{\begin{array}{l}
N'(0)\in\Lambda\su(2)\\
N'(0)(\lambda= \pm 1)=0\end{array}\right. 
\end{equation}
as $\Lambda\su(2)$ is a $\R$-vector space. Our first goal is to compute $N'(0)$ in terms of the parameters $\xx(0)$ and its derivatives. Then the derivatives of the parameters are obtained by
solving \eqref{problem-dN}.
Recall that $\Phi_t$ is the solution of $d_{\Sigma}\Phi_t=\Phi_t\eta_t$ in the universal cover $\wt{\Sigma}$
of $\C\setminus\{p_0,\cdots,p_{n-1}\}$ with initial condition $\Phi_t(0)=\Id$.
Also recall that $\gamma=\gamma_1$ is a closed curve enclosing the point $p_0=1$
and such that $\gamma(0)=\gamma(1)=0$.
By Proposition 8 in \cite{nnoids}, we have for all $t$

$$N(t)=\int_{\gamma}\Phi_t\eta_t'\Phi_t^{-1}$$
where we denote the lift of $\gamma$ to $\wt{\Sigma}$  still by $\gamma$.
Hence

$$N'(0)=\int_{\gamma}\Phi_0'\eta_0'\Phi_0^{-1}
+\Phi_0\eta_0''\Phi_0^{-1}-\Phi_0\eta_0'\Phi_0^{-1}\Phi_0'\Phi_0^{-1}.$$
Let

$$U=\Phi_0'\Phi_0^{-1}.$$
It is easy to check that (for details  compare with the proof of Proposition 8 in \cite{nnoids})

$$dU=\Phi_0\eta_0'\Phi_0^{-1}.$$
Thus

\begin{equation}
\label{eq-dN}
N'(0)=\int_{\gamma}U\,dU+\Phi_0\eta_0''\Phi_0^{-1}-dU\, U.\end{equation}
Our next goal is to compute the commutator $[U,dU]$.
Using Equations \eqref{explicit-potential}, \eqref{eq-central} and $r'(0)=0$, we compute

$$\eta_0'=
\frac{n\,dz}{2(z^n-1)}\matrix{2\lambda z^m&\lambda-\lambda^{-1}\\-4\lambda z^{2m}&-2\lambda z^m}.$$
This gives

$$dU=
\frac{n\,dz}{2(z^n-1)}\matrix{ (\lambda+\lambda^{-1})z^m&\lambda-\lambda^{-1}\\
(\lambda^{-1}-\lambda)z^{2m}&-(\lambda+\lambda^{-1})z^m}.$$
Let

$$I_k(z)=\int_0^z\frac{w^k\,dw}{w^n-1}.$$
Since $\Phi_t(0)=\Id$, we have $U(0)=0$, so integration yields

$$U=
\frac{n}{2}\matrix{(\lambda+\lambda^{-1})I_m&(\lambda-\lambda^{-1})I_0\\ (\lambda^{-1}-\lambda)I_{2m}&-(\lambda+\lambda^{-1})I_m}$$
which gives

\begin{equation}
\label{eq-commutator}
[U,dU]=\frac{n^2dz}{4(z^n-1)}\matrix{(\lambda-\lambda^{-1})^2(I_{2m}-z^{2m}I_0)&2(\lambda^2-\lambda^{-2})(I_m-z^mI_0)\\2(\lambda^2-\lambda^{-2})(z^{2m}I_m-z^mI_{2m})&(\lambda-\lambda^{-1})^2(z^{2m}I_0-I_{2m})}.
\end{equation}
To proceed, we compute the integrals involved in $\int_{\gamma}[U,dU]$.

\begin{lemma}
\label{lemma-integral}
With $J_{k,\ell}=\int_0^1\frac{x^k-x^{\ell}}{x^n-1}dx$ we have
$$\int_{\gamma}\frac{I_k\, z^{\ell}-I_{\ell}\,z^k}{z^n-1}dz=\frac{4\pi\ii}{n}J_{k,\ell}.$$
\end{lemma}

\noindent
\begin{proof}
Let $D$ be the disk bounded by $\gamma$.
Then
 
$$f_k(z)=\int_{0}^{z}\frac{w^k-w^{n-1}}{w^n-1}dw$$

\noindent
is holomorphic in $D$ because the integrand extends holomorphically to $1$,
and

$$I_k(z)=f_k(z)+\int_0^z\frac{w^{n-1}}{w^n-1}dw=f_k(z)+\frac{1}{n}\log(1-z^n).$$

\noindent
Therefore, we have

$$\int_{\gamma}\frac{I_k\,z^{\ell}-I_{\ell}\,z^k}{z^n-1}
=\int_{\gamma}\frac{f_k\,z^{\ell}-f_{\ell}\,z^k}{z^n-1}
+\frac{1}{n}\int_{\gamma}\frac{z^{\ell}-z^k}{z^n-1}\log(1-z^n).$$

\noindent
The first term on the right hand side can be computed via the Residue Theorem 

$$\int_{\gamma}\frac{f_k\,z^{\ell}-f_{\ell}\,z^k}{z^n-1}=\frac{2\pi\ii}{n}(f_k(1)-f_{\ell}(1))
=\frac{2\pi\ii}{n}J_{k,\ell}.$$

\noindent
The second term can be computed via integration by parts and then applying the Residue Theorem:

\begin{eqnarray*}
\int_{\gamma}\frac{z^{\ell}-z^k}{z^n-1}\log(1-z^n)
&=&\int_{\gamma}(f'_{\ell}-f'_k)\log(1-z^n)\\
&=&\left[(f_{\ell}-f_k)\log(1-z^n)\right]_{\gamma(0)}^{\gamma(1)}-\int_{\gamma}(f_{\ell}-f_k)\frac{nz^{n-1}}{z^n-1}\\
&=&0-2\pi\ii (f_{\ell}(1)-f_k(1)).
\end{eqnarray*}\end{proof}

\medskip

\noindent
\begin{proof}[Proof of Proposition  \ref{prop-derivatives} continued]
Using Equation \eqref{eq-commutator} and Lemma \ref{lemma-integral}, we obtain

\begin{equation}
\label{term1}
\int_{\gamma}[U,dU]=
\pi\ii\,n\matrix{(\lambda-\lambda^{-1})^2J_{2m,0}&
2(\lambda^2-\lambda^{-2})J_{m,0}\\2(\lambda^2-\lambda^{-2})J_{m,2m}&
-(\lambda-\lambda^{-1})^2J_{m,0}}, 
\end{equation}

\noindent
and by Leibniz rule we have

$$\eta''_0=\matrix{0&0\\r'' mz^{m-1}dz&0}+2\sum_{i=0}^{n-1}\matrix{a'_i&\lambda^{-1}b'_i\\\lambda c'_i&-a'_i}\frac{dz}{z-p_i}$$
where $a'$, $b'$, $c'$ are evaluated at $t=0$. By the Residue Theorem

\begin{eqnarray}
\int_{\gamma}\Phi_0\eta''_0\Phi_0^{-1}&=&
4\pi\ii\Res_1\Phi_0\matrix{a'&\lambda^{-1}b'\\\lambda c'&-a'}\Phi_0^{-1}\frac{dz}{z-1}
\nonumber\\
&=&4\pi\ii\matrix{a'-\lambda^{-1}b'&\lambda^{-1}b'\\ 2a'-\lambda^{-1}b'+\lambda c'&-a'+\lambda^{-1}b'}.
\label{term2}\end{eqnarray}
Recall from Equation \eqref{eq-dN}
that $N'(0)$ is the sum of \eqref{term1} and \eqref{term2}.
We now solve Problem \eqref{problem-dN} by the method of Section \ref{sec:implicit}.
By Proposition \ref{prop-time-parity}, $b'_{\mid\lambda=0}=0,$ so we may write $b'=\lambda\wt{b}'$.
We then have
\begin{equation}
\label{eq-NF}
0=N'_{11}+N'_{11}\mbox{}^*=4\pi\ii(a' -\wt{b}' -a'\mbox{}^*+\wt{b}'\mbox{}^*)
\end{equation}
\begin{equation}
\label{eq-NG}
0=N'_{21}+N'_{12}\mbox{}^*=4\pi\ii\left( \kappa_m(\lambda^2-\lambda^{-2})+2a'-\wt{b}'+\lambda c'-\wt{b}'\mbox{}^*\right).
\end{equation}
Projecting Equation \eqref{eq-NG} on $\cal{W}^{<0}$, $\cal{W}^{>0}$ and $\cal{W}^0$ and Equation \eqref{eq-NF} on $\cal{W}^{>0}$ we obtain
$$\left\{\begin{array}{l}
\wt{b}'\mbox{}^+=a'\mbox{}^+=-\lambda^2\kappa_m\\
\quad c'=0\\
a'\mbox{}^0=\wt{b}'\mbox{}^0\,.\end{array}\right.$$
\noindent
Then 
$$0=N'_{12}|_{\lambda=1}=4\pi\ii\wt{b}'(1)=4\pi\ii(\wt{b}'\mbox{}^0-\kappa_m)$$
gives
$$
\wt{b}'\mbox{}^0=\kappa_m$$
concluding the proof.
\end{proof}

\section{Area estimates for Lawson surfaces}\label{sec:construction}

\noindent
%Geometric properties of a minimal surface can be computed from its DPW potential. 
In this section we first compute the area in terms of the DPW potential and then show that the surfaces we construct by Proposition \ref{prop:IFT} yields Lawson surfaces for certain rational values of $t$ small enough.

\subsection{The area of a minimal surface via DPW}\label{sec:area}$\;$\\

\begin{proposition}
Let $\eta$ be a holomorphic DPW potential on a compact domain $\Omega$ such that a solution $\Phi$ of $d_\Omega\Phi=\Phi\eta$ solves the Monodromy Problem \eqref{monodromy-problem}.
Let $(F,B)$ the Iwasawa decomposition of $\Phi$
and $f$ the resulting minimal immersion in $\S^3$.
Then

\begin{equation}
\label{area-formula}
\Area(f(\Omega))=-2\ii\int_{\partial\Omega}\tr(\eta_{-1}B_0^{-1}B_1),
\end{equation}

\noindent
where $B=\sum_{k=0}^{\infty} \lambda^kB_k$ and $\eta=\sum_{k=-1}^{\infty}\lambda^k\eta_{k}$.
\end{proposition}

\noindent
\begin{proof} First, observe that $B$ is globally defined on $\Omega,$ because $\Phi$ solves the Monodromy Problem. The minimal surface $f$ comes with an associated family of flat connections given by  

$$d_\Omega + F^{-1}d_\Omega F.$$
\noindent
In a local coordinate $z$, we can split the connection $1$-form into its complex linear and anti-linear parts

$$F^{-1}d_\Omega F=U dz+Vd\z,$$

\noindent
and 
 compute (compare with \cite{Bobenko_1991,He})

$$U=\matrix{\rho^{-1}\rho_z&\lambda^{-1}\rho^2 a_{-1}\\b_0\rho^{-2}& -\rho^{-1}\rho_z}\qquad
V=\matrix{-\rho^{-1}\rho_{\z}&-\overline{b_0}\rho^{-2}\\-\lambda\rho^2\overline{a_{-1}}&\rho^{-1}\rho_{\z}},
$$
for some real valued and positive function $\rho$ with

$$
\eta_k=\matrix{c_k&a_k\\b_k&-c_k}dz.$$
Then the induced volume form $dA$ of the minimal immersion $f$ is computed to be

\begin{equation}
\begin{split}
dA&= 4\rho^4|a_{-1}|^2dx\wedge dy \\
&=-2\ii\,\tr\matrix{0&\rho^2a_{-1}\\0&0}\matrix{0&0\\-\rho^2\overline{a}_{-1}&0}dz\wedge d\z\\
&= -2\ii\,\tr(U_{-1}dz\wedge V_1d\z)
\end{split}
\end{equation}

\noindent
Let  $\partial B$ and $\overline{\partial}B$ denote the complex linear and complex anti-linear part of $dB.$ 
Then we have by \eqref{nablalambda}  

$$U_{-1}dz=B_0\eta_{-1}B_0^{-1}\quad\mbox{ and }
\quad V_{1}d\z=-\overline{\partial}B_1 B_0^{-1}+\overline{\partial} B_0\, B_0^{-1}B_1B_0^{-1}.$$
Using properties of the trace we obtain

$$\tr(U_{-1}dz\wedge V_1d\z)=\tr(-\eta_{-1}B_0^{-1}\wedge\overline{\partial}B_1
+\eta_{-1}\wedge B_0^{-1}\overline{\partial} B_0 \,B_0^{-1}B_1).$$

\noindent
Moreover, because $\eta$ is holomorphic 

$$d(\eta_{-1} B_0^{-1}B_1)=\eta_{-1}\wedge(B_0^{-1}\overline{\partial}B_0\, B_0^{-1}B_1-B_0^{-1}\overline{\partial}B1).$$

\noindent
Therefore,  

$$\Area(f(\Omega))=-2\ii\int_{\Omega}d\,\tr(\eta_{-1} B_0^{-1}B_1)=
-2\ii\int_{\partial\Omega}\tr(\eta_{-1} B_0^{-1}B_1)$$ by Stokes' theorem. 
\end{proof}

\noindent
In our case (the pull-back of) $\eta$ will have apparent singularities at $p_j$ and the corresponding boundary terms in 
Equation \eqref{area-formula} can be computed as residues.

\begin{proposition} 
Let $\eta$ be a DPW potential with a singularity at $z=p$ and $G$ a gauge such that
$\eta. G$ extends holomorphically to the disc $D(p,r)$ of radius $r>0$ around $p$. Then

$$\lim_{r\to 0} \int_{\partial D(p,r)}\tr(\eta_{-1}B_0^{-1}B_1)=-2\pi\ii\,\Res_{p}\tr(\eta_{-1}G_1G_0^{-1}).$$

\end{proposition}
\noindent
\begin{proof} Let $\wh{\eta}=\eta. G$, $\wh{\Phi}=\Phi G$ and $(\wh{F},\wh{B})$ be the Iwasawa
decomposition of $\wh{\Phi}$. Then
$$B=D\wh{B}G^{-1},$$
where $D$ is the unitary part of $G(0)$, i.e., it is a constant and diagonal matrix.
We have

\begin{equation}
\begin{split}
\eta_{-1}&=G_0\wh{\eta}_{-1}G_0^{-1},\\
B_0&=D\wh{B}_0 G_0^{-1},\\
B_1&=D(\wh{B}_1G_0^{-1}-\wh{B}_0G_0^{-1}G_1G_0^{-1}),\\
\eta_{-1}B_0^{-1}B_1&=G_0\wh{\eta}_{-1}\wh{B}_0^{-1}\wh{B}_1G_0^{-1}-\eta_{-1}G_1G_0^{-1}.\\ 
\end{split}
\end{equation}
Therefore,

$$\int_{\partial D(p,r)}\tr(\eta_{-1}B_0^{-1}B_1)=\int_{\partial D(p,r)}\tr(\wh{\eta}_{-1}\wh{B}_0^{-1}\wh{B}_1)-\int_{\partial D(p,r)}\tr(\eta_{-1}G_1 G_0^{-1}).$$
The first integral  on the right hand side goes to $0$ as $r\to 0,$ because $\wh{\eta}$ and $\wh{B}$ are smooth in
$D(p,r)$. The proposition then follows from the Residue Theorem.\end{proof}

\begin{corollary}\label{cor:areares}
Let $\Sigma$ be a compact Riemann surface and $\eta$ a DPW potential with $n$ apparent singularities at $p_0,\cdots,p_{n-1}$ solving the Monodromy problem \eqref{monodromy-problem}. Then

$$\Area(f(\Sigma))=4\pi\sum_{j=0}^{n-1}\Res_{p_j}\tr(\eta_{-1} G^j_1(G^j_0)^{-1})$$
where $G^j$ is a local gauge such that $\eta. G^j$ extends holomorphically to $p_j$.
\end{corollary}

\noindent
Example: Consider the potential $\eta$ for a great sphere and the gauge $G$ given by

$$\eta=\matrix{0&\lambda^{-1}\\0&0}dz\qquad \text{and} \qquad
G=\matrix{z&0\\-\lambda&z^{-1}}.$$
We have $\eta. G$ is holomorphic at $z=\infty$. Then

$$\Res_{\infty}\tr(\eta_{-1}G_1 G_0^{-1})=\Res_{\infty}\tr\matrix{0&dz\\0&0}\matrix{0&0\\-1&0}\matrix{z^{-1}&0\\0&z}=\Res_{\infty}\frac{-dz}{z}=1,$$
from which we obtain that the area of a great sphere is $4\pi$.
\subsection{Construction of compact minimal surfaces}\label{sec:con}$\;$\\
In Proposition \ref{prop:IFT} we have constructed a family of DPW potentials $\eta_t^{\xx(t)}$ over $\C P^1$ with $n+1$ singularities at $z=p_j, $ $j=0, ..., n-1$, and at $z=\infty$. By solving \eqref{IFTeqn}, the singularity at $z= \infty$ becomes apparent (Proposition \ref{apparent}), i.e., the corresponding minimal surface $f$ extends smoothly to $z=\infty.$ The monodromy at the other $n$ singularities $M_j(\lambda=\pm1)$ were computed in Proposition \ref{prop-monodromy-at-1}.
For $t= \tfrac{1}{2k+2}$, we obtain $M_j^{k+1}(\lambda= 1) =M_j^{k+1}(\lambda= -1)  = - \Id,$ for all $j= 0, ..., 2m+1.$
In other words, the Monodromy Problem \eqref{monodromy-problem} is solved on a $(k+1)$-fold cover of $\C P^1$ branched at $p_j$\\
%In other words, the singularities at $p_j$ become apparent, on a $(k+1)$-fold cover of $\C P^1.$\\

\noindent
Thus let $t=\frac{1}{2k+2}$ for $k\in\Z,$ $k\gg1$ in following 
and consider the compact Riemann surface $\Sigma=\Sigma_{m,k}$ of genus $g=mk$ given by the algebraic equation

\[y^{k+1}=\frac{z^{m+1}-1}{z^{m+1}+1}.\]
The
$(k+1)$-fold covering given by 

\[\pi\colon\Sigma\longrightarrow \C P^1, (y,z)\mapsto z\]
is totally branched over $p_j, j= 0, ..., 2m+1.$ 
Note that the monodromy $\mu$ (see Chapter II of \cite{Donald}) of the covering $\Sigma\to\C P^1$ is given by an element of the permutation group

\[\sigma\in\mathcal S_{k+1}\]
of order $k+1$ such that

\begin{equation}\label{eq:covmon}\mu(\gamma_{2j})=\sigma\quad \mu(\gamma_{2j+1})=\sigma^{-1},\end{equation}
for $j=0,..,m$ and simple closed curves $\gamma_j$ around the branch points $p_j$. \\

\noindent
Consider the pull-back DPW potential
$\pi^*\eta^\lambda$ on $\Sigma. $ It can be locally desingularized around the preimages of the branch points $\hat p_j = \pi^{-1}(p_j)$  as follows:
Let $w$ be a local holomorphic coordinate on $\Sigma$ centered at $\hat p_j$ such that

\[w^{k+1}=z-p_j.\]
Since $t(k+1)=\frac{1}{2}$, the residue of the connection $d+ \pi^*\eta$ at $w=0$ is
\begin{equation}\label{eq:28}t(k+1)A_j(\lambda)=\tfrac{1}{2}\matrix{ a_j(\lambda)&\lambda^{-1} b_j(\lambda)\\ \lambda c_j(\lambda)&- a_j(\lambda)}
\end{equation}
for $a_j, b_j, c_j \in \mathcal W^{\geq 0}$ as in Section \ref{sec:symm} satisfying, by Proposition \ref{prop-monodromy-at-1}

\begin{equation}
\label{eq-abc}
- a_j(\lambda)^2- b_j(\lambda)c_j(\lambda)=-1.
\end{equation}

\noindent
Consider the local gauge transformation

\begin{equation}\label{locgaugeeq}
g_j=g_j(w,\lambda)=\matrix{\frac{ b_j(\lambda)}{1-a_j(\lambda)} &0 \\ \lambda&\frac{1- a_j(\lambda)}{b_j(\lambda)}}\matrix{\frac{1}{\sqrt{w}}&0\\0&\sqrt{w}},
\end{equation}
\noindent
which is well-defined on a double covering of the $w$-disc (centered at $w=0$) and some $\lambda$-disc centered at $\lambda=0$.
A computation gives
\[\hat\eta:=\pi^*\eta. g_j=\matrix{0 & \frac{(a_j-1)^2}{2\lambda  b_j}\\
\frac{\lambda b_j(a_j^2+b_j c_j-1)}{(a_j-1)^2w^2}&0}dw
+O(w^{k-1})dw\]
which extends holomorphically to $w=0$ thanks to Equation \eqref{eq-abc}.
 Moreover, the $\lambda^{-1}$-term
of $\hat\eta$ is non-zero at $w=0$.
\begin{remark}
The gauge \eqref{locgaugeeq} is not necessarily well-defined on the 
whole $\lambda$-unit-disc. Therefore, we need to apply the $r$-Iwasawa decomposition instead of the ordinary Iwasawa decomposition for the reconstruction. This does not alter the corresponding minimal surface.
\end{remark}

\noindent
On the domain of the coordinate $w$ the minimal surfaces $\hat f$ and $f$ obtained from the DPW potentials $\hat \eta$ and $\pi^*\eta,$ respectively, coincide. Thus $f$ extends smoothly to $w=0$ \footnote{In order to see that one does actually obtain the same surface, one can first work on a double covering of the $w$-plane, and then prove that the unitary factor of the Iwasawa decomposition is already defined on the $w$-plane, while the gauge and the positive part of the Iwasawa decomposition have monodromy $-\Id$ around $w=0$.}. We have shown the following
\begin{proposition} \label{extension}
For $t= \tfrac{1}{2k+2}$ the pull-back potential $\wt\eta= \pi^*\eta^{\xx(t)}_t$ on $\Sigma_{m,k}$ has apparent singularities at $\pi^{-1}(p_j), j = 0, ..., n-1.$ In other words, the minimal immersion corresponding to the DPW potential $\wt\eta$ extends smoothly to $p_j, j= 0, ..., n-1.$
\end{proposition}

\begin{theorem}\label{thm:new_surface?}For every $m\in\N^{\geq1}$ fixed, there is a $K\in\N$ such that                                    for every $k\geq K$ there exists an immersed compact minimal surface $f_{m,k}$ of genus $g=mk$ in $\S^3$.
Moreover, the symmetry group of $f_{m,k}$ contains $\Z_{m+1}\times\Z_{k+1}$.
\end{theorem}
\begin{proof} 
Proposition \ref{prop:IFT} shows the existence of  DPW potentials $\eta_t=\eta_t^{\xx(t)}$ for $t\sim 0$. 
Thus let $K\in\N$ such that $\eta_t$ exist for all $t<\tfrac{1}{2K}.$ Fix an integer $k \geq K$
and consider $\wt \eta= \pi^*\left(\eta_{t} \right),$ $t=\tfrac{1}{2k+2}$ the pull-back DPW potential to $\Sigma_{m,k}$. \\

\noindent
Let $\Phi$ be the solution of $d_{\Sigma_{m,k}} \Phi = \Phi \wt\eta$, with initial condition $\Phi(\wt{0})=\Id$, where $\wt{0}\in\Sigma_{m,k}$ is a preimage of $z=0$ under $\pi.$ We claim that the Sym-Bobenko formula yields
a well-defined minimal immersion $f\colon \Sigma_{m,k}\to\S^3.$ \\

\noindent 
By 
Proposition \ref{prop:IFT} and Equation \eqref{eq:covmon}
the pull-back potential $\wt{\eta}$ satisfies the closing conditions on $\Sigma_{m,k}\setminus S$, where $S=\pi^{-1}\{z\mid z=\infty\text{ or } z=p_j, j = 0, ...n-1\}$. 
Indeed, the extrinsic closing condition follows from the construction of the covering $\Sigma_{m,k}\to\C P^1$:
A closed curve $\gamma$ on the $(2m+2)$-punctured sphere lifts to a closed curve $\hat\gamma$ in $\Sigma_{m,k}$ if and only if the monodromy
$\mu$  in \eqref{eq:covmon} of $\Sigma_{m,k}\to\C P^1$ along $\gamma$ is trivial. Comparing    $\mu$ 
with the monodromy representation of the potential $\eta_t$ at $\lambda=\pm1$ (see Proposition \ref{prop-monodromy-at-1}) we directly see that the monodromy of the potential $\eta_t$ at $\lambda=\pm1$ along a closed curve $\hat\gamma$ in $\Sigma_{m,k}\setminus S$ is $\pm\text{Id}.$
%the monodromy of the covering $\mu$ in \eqref{eq:covmon} of $\Sigma_{m,k}\to\CP^1$ can be identified
%with the monodromy of the potential $\eta_t$ at $\lambda=\pm1$ 
%The extrinsic closing follows from the fact that the fundamental group of the $(2m+2)$ punctured surface $\Sigma_{m,k}\setminus S$ can be identified with the kernel of the representation $\mu$ in \eqref{eq:covmon} of the fundamental group of the 4-punctured sphere.
 We therefore obtain a well-defined minimal immersion

$$f\colon\Sigma_{m,k}\setminus S\longrightarrow\S^3.$$ 

\noindent By Proposition \ref{extension} the minimal immersion $f$ extends as an immersion through the branch points $p_j$ of $\pi$. Proposition \ref{apparent} shows that the surface also extends smoothly through the preimages $\pi^{-1}(\infty)$ and we obtain a well-defined map $f_{m,k}\colon\Sigma_{m,k}\to\S^3.$\\ 

\noindent
It remains to show that $f_{m,k}$ is immersed at $\pi^{-1}(\infty)$. 
This follows either by Remark \ref{remark-imm-infty} or from the following counting argument:
On a branched minimal surface of genus $g= mk$ the Hopf differential $Q$ has $4g-4-b$ zeros (counted with multiplicity), where $b$ is the number of branch points (counted with multiplicity). On the other hand, for $f_{m,k}$ the form of the DPW potential and \eqref{explicit-potential}
gives that $Q$ is a constant multiple of
\[\pi^*\frac{z^{m-1}(dz)^2}{z^{2m+2}-1}.\]
This gives 
\[(k-1)(2m+2)+(2k+2)(m-1)=4 km-4=4g-4\]
zeros of $Q$.  Thus $b=0$ and $f_{m,k}$ must an immersion. \\

\noindent
That the surface $f_{m,k}$ has a $\Z_{m+1}$ and a $\Z_{k+1}$ symmetry follows from the symmetries of the potential and by uniqueness of the Iwasawa decomposition.
The $\Z_{m+1}$-action rotates the tangent plane of $f(\wt{0})\in \S^3\subset\R^4$ and fixes its orthogonal complement, while the
$\Z_{k+1}$-action fixes the tangent plane of $f(\wt{0})\in \S^3\subset\R^4$ and rotates its orthogonal complement.
Hence, the $\Z_{m+1}$ and  $\Z_{k+1}$-actions commute.
\end{proof}
\begin{theorem}\label{the:area}
For $k\to\infty$, the asymptotic expansion of the area of the minimal surfaces $$f_{m,k}\colon\Sigma_{m,k}\longrightarrow \S^3$$   is given by
\[\Area(f_{m,k})=4\pi (m+1)\left(1-\frac{\kappa_m}{2(k+1)}+ O\left(\frac{1}{(k+1)^3}\right)\right),\]
with $\kappa_m$ as defined in  Proposition \ref{prop-derivatives}%\eqref{def:kappa}.
\end{theorem}
\begin{proof}
Recall that the area of $f_{m,k}$ is given by a sum of residues, see Corollary \ref{cor:areares}.
The local gauges \eqref{locgaugeeq} have local monodromies $-\text{Id}$ around $\pi^{-1}(p_j)$ on $\Sigma_{m,k}$. Thus, we consider the double covering  $\hat\Sigma_{m,k}\to\Sigma_{m,k}$
defined by the $(2k+2)$-fold covering $\hat\pi$ of $\C P^1:$

\[\hat y^{2k+2}=\frac{z^{m+1}-1}{ z^{m+1}+1}.\]
Applying Corollary \ref{cor:areares} to the potential $\hat\pi^*\eta:=\hat\pi^*\eta^{\xx(t)}_{t}$ for $t=\tfrac{1}{2k+2},$ gives rise to  a minimal surface 
$$\tilde f_{m,k}: \hat\Sigma_{m,k}\longrightarrow\Sigma_{m,k}\to\S^3,$$
where $\tilde f_{m,k}$ is a double cover of $f_{m,k}.$
A direct computation (using the gauge $G_0$ in Section \ref{sec:reginf}) yields that there is no contribution of residues at the points over $z=\infty.$ We claim that at each preimage $\wh p_j$  of the branch points $p_j,$ $j= 0, ..., n-1$, the residue is

\begin{equation}\label{eq:rescom}\Res_{\wh p_j}\tr(\hat\pi^*\eta_{-1} g_{j,1}g_{j,0}^{-1})=1- a(t)|_{\lambda=0}\end{equation}
where $a(t)$ is provided by Proposition \ref{prop:IFT}, and the gauge is given by \eqref{locgaugeeq}.
Indeed, using \eqref{eq:28} and \eqref{locgaugeeq} and the coordinate $x=\sqrt{w}$ centered at $\wh p_j$
we have

\[\hat\pi^*\eta_{-1} =\matrix{0& b_j(0)\\0&0}\frac{dx}{x}+\text{ higher order terms in }x,\]

\[g_{j,1}=\matrix{\tfrac{a_j'(0)b_j(0)+b_j'(0)(1-a_j(0))}{(1-a_j(0))^2}\frac{1}{x}&0\\\frac{1}{x}&-\tfrac{a_j'(0)b_j(0)+b_j'(0)(1-a_j(0))}{b_j(0)^2}x}
%+\text{ higher order terms in }x
\]
\noindent
and

\[g_{j,0}=\matrix{\tfrac{b_j(0)}{1-a_j(0)}\frac{1}{x}&0\\0& \tfrac{1-a_j(0)}{b_j(0)}x }
%+\text{ higher order terms in }x,
\]
which yields \eqref{eq:rescom}.
(In the above computations, $t$ is fixed and therefore omitted, and $a_j$, $b_j$ are seen as functions of $\lambda$.)
By Propositions \ref{prop-derivatives} and \ref{prop-time-parity},
$$1-a(t)|_{\lambda=0}=1-\kappa_m t+O(t^3).$$
By Corollary \ref{cor:areares}
\[2\,\Area(f_{m,k})=\Area(\wt{f}_{m,k})=4\pi (2m+2)(1-\kappa_m t+O(t^3)).\]
\end{proof}
\begin{remark}
Due to time parity (Proposition \ref{prop-time-parity}) the same minimal surface is obtained when choosing $t=\frac{-1}{2(k+1)}$, though the area computation differs in detail. Indeed, the residue in \eqref{eq:28} will have the opposite sign, so the gauge
$g_j$ in \eqref{locgaugeeq} needs to be altered turning the right hand side of
\eqref{eq:rescom} into 
$1+a(t)|_{\lambda=0}.$ This gives of course the same area for the surface, since $a(t)|_{\lambda=0}$ is odd in $t$.\end{remark}

\begin{theorem}
The  minimal surfaces $f_{m,k}\colon\Sigma_{m,k}\to \S^3$ coincide with the Lawson surfaces $\xi_{m,k}$ for $k \gg1.$ In particular, the asymptotic expansion of the area of the Lawson surfaces is given by Theorem \ref{the:area}.
\end{theorem}
\begin{proof} 
 Using the symmetries we first show that the geodesic polygon of the construction of the Lawson surface is contained in the surface  $f_{m,k}$:
By construction of the potential, $\eta_t$ admits the symmetry  $\sigma^*\eta_t=\bar\eta_t$ for $\sigma(z)=\bar z,$ see Section \ref{sec:symm}.
Analogously to \cite{He}  it can be shown (for the initial value $\Phi(0)=\text{Id}$) that the line 
$$\{z\in\R\subset \C\mid  |z|<1\}$$
 is mapped via $f=f_{m,k}$ to a geodesic in the 3-sphere.  The symmetry $\delta\circ\delta$ of the surface is induced by rotational symmetry \[x\in\S^3=\SU(2)\longmapsto D^{-2}xD^2\]
 of the 3-sphere. Its fix point set is the circle
 \[C_1=\left\{\matrix{w&0\\0&\bar w}\mid w\in\S^1\subset\C\right\}.\]
 The induced symmetry $\delta\circ\delta\circ\sigma$ on the $z$-plane fixes the line 
 \[\{z\in\C\mid \text{arg}(z)=\tfrac{\pi }{m+1},|z|<1\}\]
 and therefore $f$ also maps this line to a geodesic in the 3-sphere using the same arguments as above. The analytic continuation of these two geodesics on the abstract Riemannian surface extend as geodesics in 3-space (contained in the surface) through the points $f(1)$ respectively $f(\exp{\tfrac{\pi i}{m+1}}).$
 By construction of the surface via \eqref{symbob} and Proposition \ref{prop-monodromy-at-1}, there is a rotational symmetry (induced from the monodromy around $z=1$) of the surface
 which is given by
 \[x\in\S^3=\SU(2)\longmapsto \matrix{e^{2\pi\ii t}&0\\0&e^{-2\pi\ii t}}x\matrix{e^{2\pi\ii t}&0\\0&e^{-2\pi\ii t}}\]
 with $t=\tfrac{1}{2k+2}$. The fix point set of this rotation is the circle $C_2$ orthogonal to the circle $C_1$. Clearly,
 $f(1)$ is a fixed point  of the rotational symmetry, and hence lies on $C_2.$ Analogously, we find that $f(\exp{\tfrac{\pi i}{m+1}})\in C_2$ by applying the symmetry induced from the monodromy around $z=\exp{\tfrac{\pi i}{m+1}}.$ Applying the
rotational symmetry  at the point $f(1)$ together with the reflection symmetry across the geodesic in 3-space
which contains $\{f(x)\mid x\in\R,|x|<1\}$, we easily deduce that also $\{f(x)\mid x\in\R,x>1\}$ is a geodesic in the 3-sphere.
Analogously, $f$ maps $\{z\in\C\mid \text{arg}(z)=\tfrac{\pi }{m+1},|z|>1\}$ to a geodesic. Finally, being a fix point of the rotational symmetry $\delta\circ\delta$, $f$ maps (points over) $z=\infty$ to a point on $C_1.$ The angles between
the four different geodesics joining $f(0)$ and $f(1)$, $f(1)$ and $f(\infty)$, $f(\infty)$ and $f(\exp{\tfrac{\pi i}{m+1}}),$ $f(\exp{\tfrac{\pi i}{m+1}})$ and $f(0),$ respectively, must be the same as the angles in the geodesic polygon of the Lawson surface by the very form of the symmetries.
 From these observations we obtain that
$f_{m,k}$ maps the boundary of the
sector 
$$Se=\{z\in\C\mid 0\leq\text{arg}(z)\leq \tfrac{\pi }{m+1}\}$$
 to the geodesic polygon $\Gamma$  in the construction of a Lawson surface.\\

\noindent
%Let $S:=f_{m,k}(Se)$.  
We want to prove that $f_{m,k}(Se)$ is contained in a hemisphere for $k$ large enough.
The contour $\Gamma$ is contained in a ball $B(p,r)\subset\S^3$ of radius $r<\frac{\pi}{2}$.
Assume by contradiction that there exists a point $q\in S$ such that
$d(q,p)\geq \frac{\pi}{2}$. Then $d(q,\Gamma)\geq \frac{\pi}{2}-r$.
By the Monotonicity Formula for minimal surfaces, the area of $S$ is greater than
$c(\frac{\pi}{2}-r)^2$ for some universal constant $c$. But we know that the area of $S$
is equal to $\frac{1}{n(k+1)}$ of the area of $f_{m,k}(\Sigma)$, so is less than $\frac{2\pi}{k+1}$. Hence for $k$ large enough, $S$ is included in the hemisphere
$B(p,\frac{\pi}{2})$. Then the solution of the Plateau Problem is unique by a standard
application of the maximum principle (see \cite[Theorem 4.1]{KapWiy}).
Hence $f_{m,k}(\Sigma)$ is the Lawson surface $\xi_{m,k}$ for $k\gg1$.
\end{proof}

\appendix\label{app}
\section{On removable singularities}
\noindent
The DPW method can also be applied to obtain CMC surfaces from $\Sigma$ into 3-dimensional space forms. 
In this section, we want to give sufficient conditions for a singularity of a DPW potential to be apparent in this more general setup. Be aware of the slightly differing notations in this section. \\

\noindent The Monodromy Problem associated to the general Sym-Bobenko formula at
Sym-points $\lambda_1$ and $\lambda_2 \in\C_*$ is: If $\Phi$ is a solution of $d_\Sigma \Phi=\Phi\eta$
and $M(\lambda)$ is the monodromy of $\Phi$, then

\begin{equation}
\label{general-monodromy-pb}
\mbox{ if $\lambda_2\neq \lambda_1$:}
\left\{\begin{array}{l}
M\in\Lambda SU(2)\\
M(\lambda_1)=\Id_2\\
M(\lambda_2)=\Id_2
\end{array}\right.
\qquad\qquad\mbox{ if $\lambda_2=\lambda_1$:}
\left\{\begin{array}{l}
M\in\Lambda SU(2)\\
M(\lambda_1)=\Id_2\\
M'(\lambda_1)=0\,.
\end{array}\right.
\end{equation}
\begin{example}$\;$\\

\begin{itemize}
\item $(\lambda_1,\lambda_2)=(1,1)$ produces surfaces in $\R^3$ with $H\equiv 1$,
\item $(\lambda_1,\lambda_2)=(e^{\ii\alpha},-e^{-\ii\alpha})$ produces surfaces
in $\S^3$ with $H\equiv \tan\alpha$,
\item $(\lambda_1,\lambda_2)=(e^q,e^{-q})$ produces surfaces in $\H^3$
with $H\equiv \coth q$.
\end{itemize}
\end{example}
\noindent
We assume that the Sym-points are chosen so that

$$\lambda_1+\lambda_2\in e^{\ii\theta}\R\quad\mbox{ and }\quad
\lambda_1\lambda_2=e^{2\ii\theta}$$
for some $e^{\ii\theta}\in\S^1$.
For the above examples, $e^{\ii\theta}$ is respectively $1$, $\ii$ and $-1$.
This ensures that

\begin{equation}
\label{eq-realquotient}
\frac{(\lambda-\lambda_1)(\lambda-\lambda_2)}{\lambda e^{i\theta}}\in\R, \qquad  \forall \lambda\in\S^1.
\end{equation}
Moreover,  we fix $\rho>1$ such that both Sym-points satisfy $|\lambda_1|< \rho$, $|\lambda_2|<\rho$.
\medskip

\noindent

\begin{theorem}
\label{thm:appendix}
Fix an integer $m\geq 1$.
For $t\in(-\varepsilon,\varepsilon)$ let $\eta_t$ be a family of DPW potentials on $D^*(0,r)$ and $\Phi_t$ a family of
solutions of $d_{D^*(0,r)}\Phi_t=\Phi_t\eta_t$ on its universal cover, with
 $C^1$-dependence on $(t,z)$
 as maps into $\Lambda\sl(2,\C)_{\rho}$
and $\Lambda SL(2,\C)_{\rho}$,  respectively.
Assume
\begin{enumerate}
\item $\eta_t$ has a pole of order at most $2m+1$ at $z=0$ with principal part

$$\eta_t(z,\lambda)=\matrix{0&\lambda^{-1}\\0&0}\left(\frac{a_t(\lambda)}{z^{2m+1}}+\frac{b_t(\lambda)}{z^{m+1}}+\frac{c_t(\lambda)}{z}\right)dz+\Xi_t(z,\lambda)$$
where $\Xi_t$ is holomorphic with respect to $z$ in $D(0,r)$.
\item $$\eta_0=\matrix{0&0\\m z^{m-1}dz&0}.$$
\item $\Phi_t$ solves the Monodromy Problem \eqref{general-monodromy-pb}.
\item $a_t^0= \Re(e^{-\ii\theta}b_t^0)=0$.
\end{enumerate}
Then  $a_t\equiv b_t\equiv c_t\equiv 0$, for $t$ small enough. In particular,  $\eta_t$ is holomorphic at $z=0$.
\end{theorem}

\noindent
\begin{proof} We can write

$$\Phi_0(z,\lambda)=V(\lambda)\matrix{1&0\\z^m&1}\quad\mbox{ with } V\in\Lambda SL(2,\C).$$
Let $(F,B)$ be the Iwasawa decomposition of $V$.
By Theorem 5 in \cite{minoids} we have $F\in\Lambda SU(2)_{\rho}$ and
$B\in\Lambda_+^{\R} SL(2,\C)_{\rho}$.
Replacing $\Phi_t$ by $F^{-1}\Phi_t$ for all $t$,
we can assume without loss of generality that $V\in\Lambda_+^{\R}SL(2,\C)_{\rho}$
(which does not change the monodromy properties of $\Phi_t$).
\medskip

\noindent
For $\xx=(a,b,c)\in(\cal{W}^{\geq 0})^3$, let $\eta_{t}^{\xx}$ be the potential in $D^*(0,r)$ defined
by

$$\eta_{t}^{\xx}(z,\lambda)=\matrix{0&\lambda^{-1}\\0&0}\left(\frac{a}{z^{2m+1}}+\frac{b}{z^{m+1}}+\frac{c}{z}\right)dz+\Xi_t(z,\lambda).$$
Let $\Phi_{t}^{\xx}$ be the solution of $d\Phi_{t}^{\xx}=\Phi_{t}^{\xx}\eta_{t}^{\xx}$ on the universal cover
with initial condition $\Phi_{t}^{\xx}(\wt{z}_0,\lambda)=\Phi_t(\wt{z}_0,\lambda)$.
We consider the problem of finding $\eta_t^\xx$ such that 

\begin{equation}
\label{pb-appendix}
\left\{\begin{array}{l}
\mbox{$\Phi_{t}^{\xx}$ solves the Monodromy Problem \eqref{general-monodromy-pb}}\\
a^0=0\\
\Re(e^{-\ii\theta}b^0)=0\,.\end{array}\right.\end{equation}
Writing $\xx_t=(a_t,b_t,c_t)$,
we have $\eta_t=\eta_{t}^{\xx_t}$ and $\Phi_t=\Phi_{t}^{\xx_t}$.
We want to apply an Implicit Function argument to show that
for $(t,\xx)$ in a neighbourhood of $(0,0)$, solving Problem \eqref{pb-appendix} is equivalent to $\xx_t\equiv 0$, from which Theorem \ref{thm:appendix} follows.\\

\noindent
Fix a base point $z_0\in D^*(0,r)$. Let $\wt{z}_0$ be a lift of $z_0$ to the universal cover $\wt{D^*}(0,r)$ and $\gamma$ be a generator of $\pi_1(D^*(0,r),z_0)$. Let $M(t,\xx)$ be the monodromy of $\Phi_{t}^{\xx}$ with respect
to $\gamma$. Then the following Lemma holds.\\

\begin{lemma}\label{appendix-lemma1}
$\;$ \\
\begin{enumerate}
\item $(t,\xx)\mapsto M(t,\xx)$ is a $C^1$ map from $(-\epsilon,\epsilon)\times(\cal{W}^{\geq 0})^3$ to
$\Lambda SL(2,C)_{\rho}$.
\item For all $t\in (-\epsilon,\epsilon)$, $M(t,0)=\Id_2$.
\item The partial derivative of $M$ with respect to $\xx$ at $(0,0)$ is given by
$$d_{\xx}M=\frac{2\pi\ii}{\lambda}V\matrix{-db&dc\\-da&db}V^{-1}.$$
\end{enumerate}
\end{lemma}

\noindent
\begin{proof}

\begin{enumerate}
\item Point 1 follows from standard ODE theory.
\item Point 2 follows from the fact that $\eta_{t}^{\xx=0}$ is holomorphic in $D(0,r)$.
\item Let $\Psi_{t}^{\xx}$ be the solution of $d\Psi_{t}^{\xx}=\Psi_{t}^{\xx}\eta_{t}^{\xx}$ in the universal cover with initial condition
$\Psi_{t}^{\xx}(\wt{z}_0)=\Id_2$. Then $\Phi_{t}^{\xx}=\Phi_t(\wt{z}_0)\Psi_{t}^{\xx}$ and thus
$$M(t,\xx)=\Phi_t(\wt{z}_0)\cal{M}_{\gamma}(\Psi_t^\xx)\Phi_t(\wt{z}_0)^{-1}.$$
\noindent
By Proposition 8 in \cite{nnoids}, the partial derivative of $\cal{M}_{\gamma}(\Psi_t^\xx)$
with respect to $\xx$ at $(t,\xx)=(0,0)$ is given by

$$d_{\xx}\cal{M}_{\gamma}(\Psi_t^\xx)=
\int_{\gamma}\Psi_0^0d_{\xx}\eta_t^\xx(\Psi_0^0)^{-1}.$$
\noindent
Hence since $\cal{M}_{\gamma}(\Psi_0^0)=\Id_2$:

\begin{eqnarray*}
d_{\xx}M
&=&\Phi_0(\wt{z}_0)d_{\xx}\cal{M}_{\gamma}(\Psi_t^\xx)\Phi_0(\wt{z}_0)^{-1}\\
&=&\int_{\gamma}\Phi_0d_{\xx}\eta_t^\xx\Phi_0^{-1}\\
&=&\int_{\gamma}V\matrix{1&0\\z^m&1}\matrix{0&\lambda^{-1}\\0&0}\left(\frac{da}{z^{2m+1}}+\frac{db}{z^{m+1}}+\frac{dc}{z}\right)
\matrix{1&0\\-z^m&1}V^{-1}\\
&=&2\pi\ii\, V\Res_0\matrix{-z^m&1\\-z^{2m}&z^m}\left(\frac{da}{z^{2m+1}}+\frac{db}{z^{m+1}}+\frac{dc}{z}\right)V^{-1}\\
&=&\frac{2\pi\ii}{\lambda}V\matrix{-db&dc\\-da&db}V^{-1}.
\end{eqnarray*}
\end{enumerate}\end{proof}

\noindent
We define for $(t,\xx)$ in a neighbourhood of $(0,0)$:
\begin{equation}
\begin{split}
\wt{M}(t,\xx)(\lambda)&=\frac{1}{\lambda-\lambda_1}\big(\log M(t,\xx)(\lambda)-\log M(t,\xx)(\lambda_1)\big)\quad\mbox{ if $\lambda\neq\lambda_1$}\\
\wh{M}(t,\xx)(\lambda)&=\frac{\lambda e^{\ii\theta}}{\lambda-\lambda_2}\big(\wt{M}(t,\xx)(\lambda)-\wt{M}(t,\xx)(\lambda_2)\big)\quad\mbox{ if $\lambda\neq\lambda_1,\lambda_2$}.
\end{split}
\end{equation}

\noindent
Then $\wt{M}(t,\xx)$ extends holomorphically to $\lambda=\lambda_1$ and
$\wh{M}(t,\xx)$ extends holomorphically to $\lambda=\lambda_1,\lambda_2$. Moreover,
by Proposition 5 in \cite{nodes}
$\wt{M}$ and $\wh{M}$ are smooth maps taking values in $\Lambda\sl(2,\C)_{\rho}$.

\begin{lemma}
\label{appendix-lemma2}
The Monodromy Problem \eqref{general-monodromy-pb} for $\Phi_{t}^{\xx}$ is equivalent to:
\begin{equation}
\label{general-monodromy-pb2}
\left\{\begin{array}{l}
\wh{M}(t,\xx)\in\Lambda\su(2)\\
M(t,\xx)(\lambda_1)=Id_2\\
\wt{M}(t,\xx)(\lambda_2)=0.\end{array}\right.
\end{equation}
\end{lemma}
\noindent
\begin{proof} The second equation in \eqref{general-monodromy-pb} and \eqref{general-monodromy-pb2} are the same.
The third one of \eqref{general-monodromy-pb} and \eqref{general-monodromy-pb2} are equivalent: While for $\lambda_2\neq\lambda_1$ the equation is the same, we use
for $\lambda_1=\lambda_2$ that

$$\wt{M}(t,\xx)(\lambda_1)=\frac{\partial}{\partial\lambda}M(t,\xx)(\lambda)|_{\lambda=\lambda_1}.$$
As for the first equation of \eqref{general-monodromy-pb} and \eqref{general-monodromy-pb2},  

$$\wh{M}(t,\xx)(\lambda)=\frac{\lambda e^{\ii\theta}}{(\lambda-\lambda_1)(\lambda-\lambda_2)}\log M(t,\xx)(\lambda).$$
Thus Lemma \ref{appendix-lemma2} follows from Equation \eqref{eq-realquotient}.
\end{proof}

\noindent
We introduce the auxiliary variables $(p,q,r)$ defined as functions of $(a,b,c)$ by

$$\matrix{-q&r\\-p&q}=V\matrix{-b&c\\-a&b}V^{-1}.$$
This change of variables is an automorphism of $(\cal{W}^{\geq 0})^3$ because
$V\in\Lambda_+^{\R}SL(2,\C)_{\rho}$ and hence  its entries are in $\cal{W}^{\geq 0}$.
Then

\begin{equation}
\label{eq-dxM}
d_{\xx}M=\frac{2\pi\ii}{\lambda}\matrix{-dq&dr\\-dp&dq}.
\end{equation}
Writing $V(0)=\minimatrix{\mu&\nu\\0&\mu^{-1}}$ with $\mu>0$ and $\nu\in\C$, we have

$$\matrix{-q^0&r^0\\-p^0&q^0}=
\matrix{-b^0-\frac{\nu}{\mu}a^0&\nu^2a^0+2\mu\nu b^0+\mu^2c^0\\-\frac{1}{\mu^2}a^0&b^0+\frac{\nu}{\mu}a^0}.$$
Therefore,

\begin{equation}
\label{eq-aequivp}
\left\{\begin{array}{l}
a^0=0\\
\Re(e^{-\ii\theta}b^0)=0\end{array}\right.\Leftrightarrow
\left\{\begin{array}{l}
p^0=0\\
\Re(e^{-\ii\theta}q^0)=0\,.\end{array}\right.\end{equation}
We decompose the parameter $q\in\cal{W}^{\geq 0}$ into

$$q(\lambda)=q(\lambda_1)+(\lambda-\lambda_1)\wt{q}(\lambda)\quad
\mbox{with $\wt{q}\in\cal{W}^{\geq 0}$}.$$
Then we further decompose  $\wt{q}$ into

$$\wt{q}(\lambda)=\wt{q}(\lambda_2)+(\lambda-\lambda_2)\wh{q}(\lambda)
\quad \mbox{with $\wh{q}\in\cal{W}^{\geq 0}$}.$$
This gives by Equation \eqref{eq-dxM}:

\begin{eqnarray*}
d_{\xx}\wt{M}_{11}&=&\frac{2\pi\ii}{\lambda-\lambda_1}(d_{\xx}M_{11}(\lambda)-d_{\xx}M_{11}(\lambda_1))\\
&=&\frac{-2\pi\ii}{\lambda-\lambda_1}\left(\frac{dq(\lambda_1)+(\lambda-\lambda_1)d\wt{q}(\lambda)}{\lambda}-\frac{dq(\lambda_1)}{\lambda_1}\right)\\
&=&-2\pi\ii\left(\frac{d\wt{q}(\lambda)}{\lambda}-\frac{dq(\lambda_1)}{\lambda\lambda_1}\right).
\end{eqnarray*}

\begin{eqnarray*}
d_{\xx}\wh{M}_{11}&=&\frac{-2\pi\ii\,\lambda e^{\ii\theta}}{\lambda-\lambda_2}\left(
\frac{d\wt{q}(\lambda_2)+(\lambda-\lambda_2)d\wh{q}(\lambda)}{\lambda}-\frac{dq(\lambda_1)}{\lambda\lambda_1}-\frac{d\wt{q}(\lambda_2)}{\lambda_2}+\frac{dq(\lambda_1)}{\lambda_1\lambda_2}\right)\\
&=&-2\pi\ii\,e^{\ii\theta}\left(d\wh{q}(\lambda)-\frac{d\wt{q}(\lambda_2)}{\lambda_2}
+\frac{dq(\lambda_1)}{\lambda_1\lambda_2}\right).
\end{eqnarray*}

\noindent
By decomposing the other parameters $q$ and $r$ in the same way and we obtain similar formulas for
the other entries of $d_{\xx}\wt{M}$ and $d_{\xx}\wh{M}$.
Let

\begin{equation}
\begin{split}
\cal{E}_1(t,\xx)&=\wh{M}_{11}(t,\xx)+\wh{M}_{11}(t,\xx)^*\in\cal{W}\\
\cal{E}_2(t,\xx)&=\wh{M}_{12}(t,\xx)+\wh{M}_{21}(t,\xx)^*\in\cal{W}\\
\cal{E}_3(t,\xx)&=\left(M_{11}(t,\xx)(1)-1,M_{12}(t,\xx)(1),M_{21}(t,\xx)(1)\right)\in\C^3\\
\cal{E}_4(t,\xx)&=\left(\wt{M}_{11}(t,\xx)(1),\wt{M}_{12}(t,\xx)(1),\wt{M}_{21}(t,\xx)(1)\right)\in\C^3.
\end{split}
\end{equation}
\noindent
The Monodromy Problem \eqref{general-monodromy-pb2} is then equivalent to
$\cal{E}_k(t,\xx)=0$ for $1\leq k\leq 4$.
To put everything together we define

$$\cal{F}(t,\xx)=\left[\cal{E}_1^+,\cal{E}_2^+,(\cal{E}_2^-)^*,\cal{E}_3,\cal{E}_4,\cal{E}_2^0,p^0,\Re(e^{-\ii\theta}q^0)+\ii \Re(\cal{E}_1^0)\right](t,\xx)
\quad\in(\cal{W}^{>0})^3\times\C^9.$$
By Equation \eqref{eq-aequivp}, Problem \eqref{pb-appendix} is equivalent to $\cal{F}(t,\xx)=0$.
Indeed, since $\cal{E}_1=\cal{E}_1^*$ we have 
$\cal{E}_1=0$ is equivalent to
$\cal{E}_1^+=0$ and $\Re(\cal{E}_1^0)=0$.
\begin{lemma}
The derivative of $\cal{F}$ at $(0,0)$ with respect to the parameters

$$(\wh{p}^+,\wh{q}^+,\wh{r}^+,p(\lambda_1),q(\lambda_1),r(\lambda_1),
\wt{p}(\lambda_2),\wt{q}(\lambda_2),\wt{r}(\lambda_2),\wh{p}^0,\wh{q}^0,\wh{r}^0)$$
is an $\R$-linear automorphism of $(\cal{W}^{>0})^3\times\C^9$.
\end{lemma}
\noindent
\begin{proof}
We have

$$d(\cal{E}_1^+,\cal{E}_2^+,(\cal{E}_2^-)^*)=2\pi\ii\,e^{\ii\theta}(-d\wh{q}^+,d\wh{r}^+,-d\wh{p}^+),$$
so the derivative of $(\cal{E}_1^+,\cal{E}_2^+,(\cal{E}_2^-)^*)$ with respect to $(\wh{q}^+,\wh{r}^+,
\wh{p}^+)$ is an automorphism of $(\cal{W}^{>0})^3$.
Let $L$ be the partial derivative of the remaining components of $\cal{F}$ with respect to the remaining variables.
If suffices to prove that $L$ is an automorphism of $\C^9$. Let

$$X=(P(\lambda_1),Q(\lambda_1),R(\lambda_1),\wt{P}(\lambda_2),\wt{Q}(\lambda_2),\wt{R}(\lambda_2),\wh{P}^0,\wh{Q}^0,\wh{R}^0)\in\mbox{Ker}(L).$$
Then

$$d\cal{E}_3 X=2\pi\ii(-Q(\lambda_1),R(\lambda_1),-P(\lambda_1))\quad\Rightarrow\quad P(\lambda_1)=Q(\lambda_1)=R(\lambda_1)=0,$$

$$d\cal{E}_4 X=\frac{2\pi\ii}{\lambda_2}(-\wt{Q}(\lambda_2),\wt{R}(\lambda_2),-\wt{P}(\lambda_2))
\quad\Rightarrow\quad \wt{P}(\lambda_2)=\wt{Q}(\lambda_2)=\wt{R}(\lambda_2)=0.$$
From

$$p(\lambda)=p(\lambda_1)+(\lambda-\lambda_1)\wt{p}(\lambda_2)+(\lambda-\lambda_1)(\lambda-\lambda_2)\wh{p}(\lambda)$$
we obtain

$$p^0=p(\lambda_1)-\lambda_1\wt{p}(\lambda_2)+e^{2\ii\theta}\wh{p}^0.$$
Hence

$$\wh{P}^0=0\quad\mbox{ and } \quad\Re(e^{\ii\theta}\wh{Q}^0)=0.$$
Then

$$d\cal{E}_2^0 X=2\pi\ii\,e^{\ii\theta}\wh{R}^0\quad\Rightarrow\quad\wh{R}^0=0$$
$$\Re(d\cal{E}_1^0 X)=4\pi \Im(e^{\ii\theta}\wh{Q}^0)\quad\Rightarrow \wh{Q}^0=0.$$
Hence $X=0$ so $L$ is an automorphism of $\C^9$.

\end{proof}

\noindent
By Implicit Function Theorem, Problem
\eqref{pb-appendix} uniquely determines $\xx$ as a function of $t$ for $(t,\xx)$ in a neighbourhood of $(0,0)$.
By Point 2 of Lemma \ref{appendix-lemma1} the unique solution is given by $\xx\equiv0$,
proving Theorem \ref{thm:appendix}.\end{proof}

\end{document}